\documentclass[12pt,a4paper]{amsart}

\usepackage[a4paper,left=2.25cm,right=2.25cm,top=3cm,bottom=3cm]{geometry}
\usepackage{palatino}
\usepackage{euscript} 
\usepackage{mathrsfs} 
\usepackage{amssymb,mathtools}
\usepackage{xcolor}
\usepackage{tikz}
\usetikzlibrary{decorations.pathreplacing}
\usepackage{hyperref}

\numberwithin{equation}{section}

\allowdisplaybreaks[1]

\newcommand\Sym{{\mathscr S}}
\newcommand\ZZ{{\mathbb Z}}
\tikzset{
  perm dot/.style={circle,draw=black,fill=black,inner sep=5pt,
    outer sep=0pt,transform shape},
  new dot/.style={circle,draw=red!65!black,fill=red!65!black,inner sep=5pt,
    outer sep=0pt,transform shape}
}

\newcommand{\Av}{\operatorname{Av}}
\newcommand{\D}{\mathscr D}
\newcommand{\C}{\mathscr C}
\newcommand{\Dtwo}{\mathscr D^{(2)}}
\newcommand{\Dthree}{\mathscr D^{(3)}}
\newcommand{\Ctwo}{\mathscr C^{(2)}}
\newcommand{\Cthree}{\mathscr C^{(3)}}

\newtheorem{theorem}{Theorem}[section]
\newtheorem{proposition}[theorem]{Proposition}
\newtheorem{lemma}[theorem]{Lemma}
\newtheorem{corollary}[theorem]{Corollary}
\newtheorem{remark}[theorem]{Remark}

\definecolor{remcolor}{RGB}{50, 200, 210}
\newcommand{\rem}[1]{}

\title[$q$-Difference Equations for two classes of Pattern Avoiding Permutations]{$q$-Difference Equations for two classes of\\Pattern Avoiding Permutations}
\author{Paul Zinn-Justin}
\address{Paul Zinn-Justin, School of Mathematics and Statistics, The University of Melbourne, 
Victoria 3010, $\hbox{Australia}$}
\email{pzinn@unimelb.edu.au}
\thanks{The author was supported by ARC grants DP210103081 and DP240101787. He 
is grateful for the hospitality of the Erwin Schr\"odinger International Institute where this research was initiated at the Workshop on Algorithmic and Enumerative Combinatorics in 2017. He would like to thank J.~Pantone for discussions in the framework of a parallel collaboration,
and for pointing out reference \cite{BMBP} to him.}
\date{}

\begin{document}

\begin{abstract}
We study the nested permutation classes
\[
 \Dthree=\Av(4123,4231,4312)\subset
 \Dtwo=\Av(4123,4312).
\]
We relate the corresponding generating functions $D^{(2)}(z)$, $D^{(3)}(z)$ to $q$-difference equations.
Solving these equations allows to show that the generation functions are not D-finite,
and that
\begin{align*}
 [z^n]D^{(3)}(z)&\sim 0.0189672325071988\ldots\,
 (4.46840899160393\ldots)^n
\\
 [z^n]D^{(2)}(z)&\sim0.0849833632890843\ldots\,
 (3+2\sqrt2)^n n^{-3/2}.
\end{align*}
\end{abstract}

\maketitle

\section{Introduction}\label{sec:introduction}
For a set $B$ of permutations, let $\Av(B)$ denote the class of
permutations avoiding every member of $B$. This paper studies the ordinary generating functions
\[
 D^{(j)}(z)=\sum_{n\geq0}d_n^{(j)}z^n,
 \qquad d_n^{(j)}=|\D^{(j)}\cap \Sym_n|,
 \qquad j\in\{2,3\},
\]
in the two cases
\[
 \Dtwo=\Av(4123,4312),
 \qquad
 \Dthree=\Av(4123,4231,4312).
\]

The first terms are
\begin{align*}
 (d_n^{(2)})_{n\geq0}
 &=1,1,2,6,22,89,382,1711,7922,37663,\ldots,\\
 (d_n^{(3)})_{n\geq0}
 &=1,1,2,6,21,79,310,1251,5150,21517,\ldots.
\end{align*}
Albert, Homberger, Pantone, Shar and Vatter~\cite{AHPSV} used restricted
containers to generate long initial segments of both sequences. Their
differential-approximant analysis suggested that
\[
 d_n^{(2)}\sim C(3+2\sqrt2)^n n^{-3/2}
\]
for some $C>0$, while the data for $d_n^{(3)}$ suggested pure exponential
growth. They also singled out these generating functions as apparently
non-D-finite, and conjectured the stronger property of differential
transcendence.

These examples have remained natural small test cases for the boundary
between exact and experimental enumeration. Garrabrant and Pak~\cite{GP}
disproved the Noonan--Zeilberger conjecture by constructing permutation
classes with non-P-recursive counting sequences, making compact explicit
examples particularly attractive; see Vatter~\cite{Vatter} for a survey
of the surrounding questions.

Recently, Bousquet-M\'elou, Bouvel and Pantone~\cite{BMBP} announced exact formulae for the
complement-equivalent class $\Av(1243,1324,1432)$; our
solution bears some similarity to those, though the exact relation deserves further investigation.
\rem{read that extended abstract}

We treat the two classes in parallel. For each
class, insertion of a new maximum is controlled by a distinguished suffix
belonging to a simpler container class. Two integer parameters of that
suffix lead to refined enumeration, with polynomials $B_n^{(j)}(x,y)$,
that satisfy a recurrence relation. Equivalently, the trivariate series
\[
 \mathbb G^{(j)}(z;x,y)=\sum_{n\geq1}B_n^{(j)}(x,y)z^{n-1}.
\]
satisfy linear functional equations which are very similar in the two cases.
In particular, these functional equations have the
same kernel
$
 (1-y)(1-zx)+zy
$.
Its root defines a M\"obius transformation which, after a change of
variables, becomes the dilation $w\mapsto q w$ and leads to $q$-difference
equations. This shared setup is
the conceptual reason to treat the two problems together.

The analytic behaviour of the two generating series is nevertheless quite different.
In the three-pattern
case, coefficientwise iteration away from the attracting fixed point
produces analytic numerator and denominator functions in the unit disc as
limits of a second-order recurrence. The denominator has a theta-type
oscillation near $q=1$. In the two-pattern
case, iteration toward the attracting fixed point produces a convergent
product plus a convergent sum; its first obstruction is the collision of
two roots of a quadratic coefficient. Thus, these very similar $q$-difference equations
lead to either a meromorphic, pure-exponential regime or an algebraic branch cut,
$n^{-3/2}$ regime.

We now state the main results.
Recall that a formal power series is \emph{D-finite} if it satisfies a
nonzero linear differential equation with polynomial coefficients. A
sequence is \emph{P-recursive} if it satisfies a nonzero linear recurrence
with polynomial coefficients; a sequence is P-recursive if and only if
its ordinary generating function is D-finite~\cite{StanleyDF}.
The numerical constants appearing in the two
statements are defined and evaluated in \S\ref{sec:three-asymptotics}
and \S\ref{sec:two-asymptotics}, respectively.

\begin{theorem}\label{thm:main-three}
The generating function $D^{(3)}$ has infinitely many finite singularities
and is not D-finite;
consequently, $(d_n^{(3)})$ is not P-recursive.
For some $r^{(3)}>\rho^{(3)}$,
\[
 d_n^{(3)}=C^{(3)}(\rho^{(3)})^{-n}+O((r^{(3)})^{-n}),
\]
where
\begin{align*}
 \rho^{(3)}&=0.2237933013470754447\ldots,\\
 (\rho^{(3)})^{-1}&=4.468408991603930817\ldots,\\
 C^{(3)}&=0.018967232507198778\ldots.
\end{align*}
\end{theorem}

\begin{theorem}\label{thm:main-two}
\[
 d_n^{(2)}\sim C^{(2)}(3+2\sqrt2)^n n^{-3/2},
 \qquad
 C^{(2)}=0.08498336328908425176\ldots.
\]
so that the generating function $D^{(2)}$ has radius of convergence
$
 \rho^{(2)}=3-2\sqrt2
$.
Moreover, analytic continuation of $D^{(2)}$ has infinitely many finite
poles. In particular, $D^{(2)}$ is not D-finite and $(d_n^{(2)})$ is not
P-recursive.
\end{theorem}

The paper is organised as follows. \S\ref{sec:common-method}
develops the combinatorial construction and kernel reduction that is common to both classes.
\S\ref{sec:three-pattern} treats the three-pattern
class, proves the main theorem in that case and gives a simple probabilistic application.
\S\ref{sec:two-pattern} treats the two-pattern class similarly; though the solution and analysis of the $q$-difference equation
is slightly simpler than in the previous case,
concluding on non D-finiteness involves a trickier analytic continuation argument.

\section{The common trivariate generating function method}\label{sec:common-method}
\subsection{Succession rules}\label{sec:succession}
We describe in this section ``succession rules'', in the terminology of the literature on generating
trees. (Similar ideas appear in the statistical mechanics literature under the name of ``transfer matrix'',
see for example \cite{CGZJ} for a transfer matrix enumeration
of a different class of pattern-avoiding permutations.)

Given a permutation class $\mathscr X$,
define the children of $\pi\in\mathscr X\cap \Sym_n$ to be the permutations in
$\mathscr X\cap \Sym_{n+1}$ obtained by inserting the new maximum $n+1$ into
$\pi$. Every nonempty permutation has a unique parent, obtained by
deleting its maximum. By a \emph{succession rule} we mean a labelling of
this tree by states such that the multiset of states of the children
depends only on the state of the parent.

Set
\[
 \Ctwo=\Av(123,312),
 \qquad
 \Cthree=\Av(123,231,312).
\]
The basis theorem for restricted-container machines~\cite{AHPSV}
identifies the $\C^{(j)}$-machine output with $\D^{(j)}$. The exact succession
rules can also be proved directly, and only the following elementary
descriptions of the container classes are needed.

Write $\delta_s=s(s-1)\cdots1$, allowing $\delta_0$ to denote an empty
block.


\begin{lemma}\label{lem:container-structure}
Every nonempty member of $\Cthree$ has a unique expression
\begin{equation}\label{eq:Cthree-blocks}
 \delta_\ell\oplus\delta_{k+1},
 \qquad k,\ell\geq0.
\end{equation}
Every nonempty member of $\Ctwo$ has an expression
\begin{equation}\label{eq:Ctwo-blocks}
 231[\delta_\ell,\delta_j,\delta_i],
 \qquad \ell,i\geq0,\quad j\geq1.
\end{equation}
In the second expression, $\ell$ is the number of entries before the
maximum and
$
 k=i+j-1
$
is the number after it. Thus $(k,\ell)$ is well-defined even when the
three-block expression is not unique.
\end{lemma}
\begin{proof}
Let $M$ be the maximum. In either class, the entries before $M$ are
decreasing, since an ascent followed by $M$ would form $123$, and the
entries after $M$ are decreasing, since $M$ followed by an ascent would
form $312$.

For $\Cthree$, an entry $a$ before $M$ and an entry $b$ after $M$ must
satisfy $a<b$, since otherwise $aMb$ forms $231$. This gives
\eqref{eq:Cthree-blocks}, and the converse is immediate.

For $\Ctwo$, no entry after $M$ can lie strictly between two entries before
$M$: those two entries, in decreasing order, followed by the intermediate
entry would form $312$. Hence the decreasing segment after $M$ consists
first of entries above all entries preceding $M$, and then of entries
below all of them. This is precisely \eqref{eq:Ctwo-blocks}. The
assertions about $k$ and $\ell$ follow by locating $M$, the first entry of
the middle block.
\end{proof}

For $\pi\in\D^{(j)}$, let $\sigma^{(j)}(\pi)$ be its longest terminal segment
whose standardisation belongs to $\C^{(j)}$. Define the state
$(k^{(j)}(\pi),\ell^{(j)}(\pi))$ by Lemma~\ref{lem:container-structure}, and
\[
 B_n^{(j)}(x,y)=\sum_{\pi\in\D^{(j)}\cap \Sym_n}
 x^{k^{(j)}(\pi)}y^{\ell^{(j)}(\pi)}.
\]
Then $B_n^{(j)}(1,1)=d_n^{(j)}$ and $B_1^{(j)}=1$.

For $\pi\in\D^{(j)}$, encode the multiset of states of the permutations
obtained by inserting a new maximum into $\pi$ by the
polynomial
\begin{equation}\label{eq:child-inventory-def}
 \mathcal I_\pi^{(j)}(x,y)
  =\sum_{\tau\text{ child of }\pi}
 x^{k^{(j)}(\tau)}y^{\ell^{(j)}(\tau)}.
\end{equation}

\begin{proposition}\label{prop:inventories}
If $\pi\in\Dthree$ has state
$(k^{(3)}(\pi),\ell^{(3)}(\pi))=(k,\ell)$, then
\begin{align}
 \mathcal I_\pi^{(3)}(x,y)
 &=y^{k+1}+\sum_{h=1}^k x^h+x^{k+1}y^\ell
   +\sum_{h=1}^\ell x^ky^h.
 \label{eq:inventory-three}
\end{align}
If $\pi\in\Dtwo$ has state
$(k^{(2)}(\pi),\ell^{(2)}(\pi))=(k,\ell)$, then
\begin{align}
 \mathcal I_\pi^{(2)}(x,y)
 &=y^{k+1}+\sum_{h=1}^k x^hy^{k+1-h}+x^{k+1}y^\ell
   +\sum_{h=1}^\ell x^ky^h.
 \label{eq:inventory-two}
\end{align}
In particular, either expression 
depends on $\pi$ only through its
state, so defines a succession rule.
\end{proposition}
\begin{proof}
Let $\pi$ have size $n$ and belong to the class under consideration.
Insert the new maximum $n+1$ into $\pi$. Any new
basis occurrence must use $n+1$. Since the maximum is first in every
basis permutation, an insertion is legal precisely when the entries to
its right lie in $\C^{(j)}$. Thus the legal sites are the $k+\ell+2$ gaps in
and at the ends of the distinguished suffix.

The four families are listed below; $h$ counts the
old entries left to the right of the insertion site.
\[
\begin{array}{c|c|c}
\text{site}&\text{new state in }\Dthree&\text{new state in }\Dtwo\\ \hline
\text{after the suffix}&(0,k+1)&(0,k+1)\\
\text{within the post-maximum part}&(h,0)&(h,k+1-h),\ 1\leq h\leq k\\
\text{immediately before the old maximum}&(k+1,\ell)&(k+1,\ell)\\
\text{within the pre-maximum part}&(k,h)&(k,h),\ 1\leq h\leq\ell.
\end{array}
\]

Figures~\ref{fig:three-succession} and \ref{fig:two-succession} show the
same four families at the level of permutation plots. A red point denotes
the newly inserted maximum when it belongs to the new distinguished
suffix; if it does not, only that suffix is drawn.

\begin{figure}[ht]
\centering
\begin{tikzpicture}[scale=0.28,baseline=(current bounding box.center)]
  \draw[gray!55,very thin] (0,0) grid (6,6);
  \foreach[count=\i] \j in {3,2,1,6,5,4}
    \node[perm dot] at (\i-.5,\j-.5) {};
  \draw[decorate,decoration=brace] (3,-.55)--node[below=2pt] {$\ell$}(0,-.55);
  \draw[decorate,decoration=brace] (6,-.55)--node[below=2pt] {$k+1$}(3,-.55);
\end{tikzpicture}
\quad $\longrightarrow$ \quad
\begin{tabular}{c@{\qquad}c@{\qquad}c@{\qquad}c}
\begin{tikzpicture}[scale=0.28,baseline=(current bounding box.center)]
  \draw[gray!55,very thin] (0,0) grid (4,4);
  \foreach \i/\j in {1/3,2/2,3/1}
    \node[perm dot] at (\i-.5,\j-.5) {};
  \node[new dot] at (3.5,3.5) {};
  \draw[decorate,decoration=brace] (3,-.55)--node[below=2pt] {$k+1$}(0,-.55);
\end{tikzpicture}
&
\begin{tikzpicture}[scale=0.28,baseline=(current bounding box.center)]
  \draw[gray!55,very thin] (0,0) grid (3,3);
  \node[new dot] at (.5,2.5) {};
  \foreach \i/\j in {2/2,3/1}
    \node[perm dot] at (\i-.5,\j-.5) {};
  \draw[decorate,decoration=brace] (3,-.55)--node[below=2pt] {$h+1$}(0,-.55);
\end{tikzpicture}
&
\begin{tikzpicture}[scale=0.28,baseline=(current bounding box.center)]
  \draw[gray!55,very thin] (0,0) grid (7,7);
  \foreach \i/\j in {1/3,2/2,3/1,5/6,6/5,7/4}
    \node[perm dot] at (\i-.5,\j-.5) {};
  \node[new dot] at (3.5,6.5) {};
  \draw[decorate,decoration=brace] (3,-.55)--node[below=2pt] {$\ell$}(0,-.55);
  \draw[decorate,decoration=brace] (7,-.55)--node[below=2pt] {$k+2$}(3,-.55);
\end{tikzpicture}
&
\begin{tikzpicture}[scale=0.28,baseline=(current bounding box.center)]
  \draw[gray!55,very thin] (0,0) grid (5,5);
  \foreach[count=\i] \j in {2,1,5,4,3}
    \node[perm dot] at (\i-.5,\j-.5) {};
  \draw[decorate,decoration=brace] (2,-.55)--node[below=2pt] {$h$}(0,-.55);
  \draw[decorate,decoration=brace] (5,-.55)--node[below=2pt] {$k+1$}(2,-.55);
\end{tikzpicture}
\\[-1pt]
$(0,k+1)$ & $(h,0)$ & $(k+1,\ell)$ & $(k,h)$
\end{tabular}
\caption{The four insertion families for $\Dthree$. The parent suffix has
the two-block form $\delta_\ell\oplus\delta_{k+1}$; the braces in the
children record block sizes.}
\label{fig:three-succession}
\end{figure}
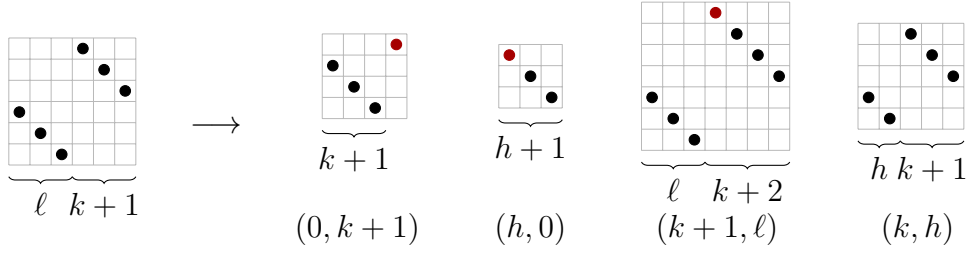

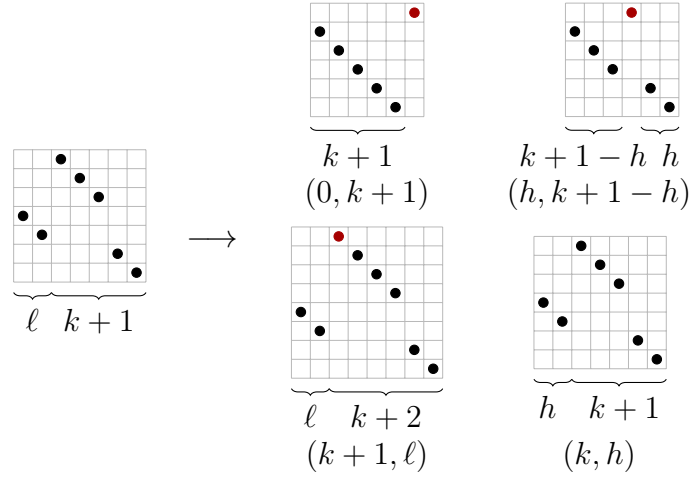
\begin{figure}[ht]
\centering
\begin{tikzpicture}[scale=0.25,baseline=(current bounding box.center)]
  \draw[gray!55,very thin] (0,0) grid (7,7);
  \foreach \i/\j in {1/4,2/3,3/7,4/6,5/5,6/2,7/1}
    \node[perm dot] at (\i-.5,\j-.5) {};
  \draw[decorate,decoration=brace] (2,-.55)--node[below=2pt] {$\ell$}(0,-.55);
  \draw[decorate,decoration=brace] (7,-.55)--node[below=2pt] {$k+1$}(2,-.55);
\end{tikzpicture}
\quad $\longrightarrow$ \quad
\begin{tabular}{c@{\qquad}c}
\begin{tikzpicture}[scale=0.25,baseline=(current bounding box.center)]
  \draw[gray!55,very thin] (0,0) grid (6,6);
  \foreach[count=\i] \j in {5,4,3,2,1}
    \node[perm dot] at (\i-.5,\j-.5) {};
  \node[new dot] at (5.5,5.5) {};
  \draw[decorate,decoration=brace] (5,-.55)--node[below=2pt] {$k+1$}(0,-.55);
\end{tikzpicture}
&
\begin{tikzpicture}[scale=0.25,baseline=(current bounding box.center)]
  \draw[gray!55,very thin] (0,0) grid (6,6);
  \foreach \i/\j in {1/5,2/4,3/3,5/2,6/1}
    \node[perm dot] at (\i-.5,\j-.5) {};
  \node[new dot] at (3.5,5.5) {};
  \draw[decorate,decoration=brace] (3,-.55)--node[below=2pt,xshift=-4pt] {$k+1-h$}(0,-.55);
  \draw[decorate,decoration=brace] (6,-.55)--node[below=2pt,xshift=4pt] {$h$}(4,-.55);
\end{tikzpicture}
\\[-1pt]
$(0,k+1)$ & $(h,k+1-h)$
\\[5pt]
\begin{tikzpicture}[scale=0.25,baseline=(current bounding box.center)]
  \draw[gray!55,very thin] (0,0) grid (8,8);
  \foreach \i/\j in {1/4,2/3,4/7,5/6,6/5,7/2,8/1}
    \node[perm dot] at (\i-.5,\j-.5) {};
  \node[new dot] at (2.5,7.5) {};
  \draw[decorate,decoration=brace] (2,-.55)--node[below=2pt] {$\ell$}(0,-.55);
  \draw[decorate,decoration=brace] (8,-.55)--node[below=2pt] {$k+2$}(2,-.55);
\end{tikzpicture}
&
\begin{tikzpicture}[scale=0.25,baseline=(current bounding box.center)]
  \draw[gray!55,very thin] (0,0) grid (7,7);
  \foreach \i/\j in {1/4,2/3,3/7,4/6,5/5,6/2,7/1}
    \node[perm dot] at (\i-.5,\j-.5) {};
  \draw[decorate,decoration=brace] (2,-.55)--node[below=2pt,xshift=-2pt] {$h$}(0,-.55);
  \draw[decorate,decoration=brace] (7,-.55)--node[below=2pt,xshift=2pt] {$k+1$}(2,-.55);
\end{tikzpicture}
\\[-1pt]
$(k+1,\ell)$ & $(k,h)$
\end{tabular}
\caption{The corresponding insertion families for $\Dtwo$. The parent is
of the form $231[\delta_\ell,\delta_j,\delta_i]$, with $k=i+j-1$.
The upper-right plot displays the mixed state $(h,k+1-h)$ responsible for
the difference between the two succession rules.}
\label{fig:two-succession}
\end{figure}

The rightmost site produces state $(0,k+1)$, and the site immediately
before the old maximum produces $(k+1,\ell)$. A site in or immediately
before the $\ell$ entries preceding the old maximum, leaving $h$ of them
to its right, produces $(k,h)$ for $1\leq h\leq\ell$. These three common
families give the first, third and fourth terms of
\eqref{eq:inventory-three} and \eqref{eq:inventory-two}.

Only insertion within the $k$ entries after the old maximum differs. If
$h$ entries remain to the right, the new distinguished $\Cthree$-suffix
starts at the new maximum and has state $(h,0)$. In $\Ctwo$, the same
suffix may retain the entries between the old and new maxima; its state is
$(h,k+1-h)$. These are respectively the second terms of
\eqref{eq:inventory-three} and \eqref{eq:inventory-two}. The maximality
of the old suffix, or deletion of the newly inserted maximum, shows that
none of the asserted new suffixes can extend farther to the left.
\end{proof}

\begin{corollary}\label{cor:B-recurrences}
For $n\geq1$,
\begin{align*}
 B_{n+1}^{(3)}(x,y)
 ={}&\left(x-\frac{y}{1-y}\right)B_n^{(3)}(x,y)
 -\frac{x-y}{(1-x)(1-y)}B_n^{(3)}(x,1)\notag\\
 &{}+yB_n^{(3)}(y,1)+\frac{x}{1-x}B_n^{(3)}(1,1),
 \\
 B_{n+1}^{(2)}(x,y)
 ={}&\left(x-\frac{y}{1-y}\right)B_n^{(2)}(x,y)\notag\\
 &{}+y\left(\frac1{1-y}+\frac{x}{x-y}\right)B_n^{(2)}(x,1)
 -\frac{y^2}{x-y}B_n^{(2)}(y,1).
\end{align*}
\end{corollary}
\begin{proof}
  Follows immediately from Proposition~\ref{prop:inventories}
  by summing \eqref{eq:child-inventory-def} over $\pi\in \D^{(j)}\cap \Sym_n$.
\end{proof}

Starting from $B_1^{(j)}=1$, the first three state polynomials are
\begin{align*}
  &B_1^{(3)}=B_1^{(2)}=1
  \\
 &B_2^{(3)}=B_2^{(2)}=x+y,\\
 &B_3^{(3)}=x^2+xy+y^2+x+2y,
 &B_3^{(2)}&=x^2+2xy+y^2+2y.
\end{align*}

\subsection{The common affine kernel}\label{sec:common-kernel}
For $j\in\{2,3\}$ define
\begin{equation*}\label{eq:specialization-hierarchy}
 \mathbb G^{(j)}(z;x,y)=\sum_{n\geq1}B_n^{(j)}(x,y)z^{n-1},
\quad
 G^{(j)}(z;x)=\mathbb G^{(j)}(z;x,1),
\quad
 g^{(j)}(z)=G^{(j)}(z;1)=\mathbb G^{(j)}(z;1,1).
\end{equation*}
Then
\begin{equation}\label{eq:Dj-from-gj}
 D^{(j)}(z)=1+zg^{(j)}(z).
\end{equation}

\begin{proposition}\label{prop:two-functional-equations}
The trivariate three-pattern series $\mathbb G^{(3)}$ satisfies the functional equation
\begin{align}
 &\bigl((1-y)(1-zx)+zy\bigr)\mathbb G^{(3)}(z;x,y)\notag\\
 &\quad=(1-y)-\frac{z(x-y)}{1-x}G^{(3)}(z;x)
 +zy(1-y)G^{(3)}(z;y)+\frac{zx(1-y)}{1-x}g^{(3)}(z),
 \label{eq:mathbbG-three-functional}
\end{align}
whereas the two-pattern series $\mathbb G^{(2)}$ satisfies
\begin{align}
 &\bigl((1-y)(1-zx)+zy\bigr)\mathbb G^{(2)}(z;x,y)\notag\\
 &\quad=(1-y)+\frac{zy(2x-y-xy)}{x-y}G^{(2)}(z;x)
 -\frac{zy^2(1-y)}{x-y}G^{(2)}(z;y).
 \label{eq:mathbbG-two-functional}
\end{align}
Both are identities of formal power series in $z$ with polynomial
coefficients after removable factors are cancelled.
\end{proposition}
\begin{proof}
Write each recurrence in Corollary~\ref{cor:B-recurrences} as
$B_{n+1}^{(j)}=\mathcal L^{(j)}B_n^{(j)}$. Since $B_1^{(j)}=1$, summation gives
$\mathbb G^{(j)}=1+z\mathcal L^{(j)}\mathbb G^{(j)}$. Multiplication by
$1-y$ and collection of terms gives
\eqref{eq:mathbbG-three-functional} and
\eqref{eq:mathbbG-two-functional}.
\end{proof}

\eqref{eq:mathbbG-three-functional} and \eqref{eq:mathbbG-two-functional} have the same kernel,
i.e., the same coefficient in front of $\mathbb G^{(j)}(z;x,y)$.
This common kernel is affine in $y$, with root
\begin{equation}\label{eq:F-def}
 Y(x)=\frac{1-zx}{1-z-zx}.
\end{equation}
The dependence on $z$ is suppressed from the notation; the
expression above shows that
\[
 Y(x)\in\mathbb Q[x][[z]],
 \qquad Y(x)=1+z+O(z^2).
\]

\begin{proposition}\label{prop:kernel-equations}
Kernel cancellation in the three-pattern equation gives
\begin{align}
0={}&\bigl(z(x+1)-1\bigr)
 \bigl(1-x+zxg^{(3)}(z)\bigr)\notag\\
 &{}-(zx^2-x+1)\bigl(z(x+1)-1\bigr)G^{(3)}(z;x)\notag\\
 &{}+z(1-x)(zx-1)G^{(3)}(z;Y(x)).
\label{eq:kernel-three}
\end{align}
For the two-pattern equation it gives
\begin{align}
0={}&\bigl(z(x+1)-1\bigr)(zx^2-x+1)\notag\\
 &{}+(zx-1)\bigl(z(x+1)-1\bigr)
 \bigl(zx^2+(z-1)x+1\bigr)G^{(2)}(z;x)\notag\\
 &{}-z(zx-1)^2G^{(2)}(z;Y(x)).
\label{eq:kernel-two}
\end{align}
\end{proposition}
\begin{proof}
The substitution $y=Y(x)$ in $\mathbb G^{(j)}(z;x,y)$ makes sense
as a $z$-adic substitution: in
$\mathbb G^{(j)}=\sum_{n\geq1}B_n^{(j)}(x,y)z^{n-1}$, each
$B_n^{(j)}$ is a polynomial in $x,y$, and the $n$th summand after
substitution still has $z$-adic order at least $n-1$.

Applying it to the appropriate functional equation, and use of
\[
 1-Y(x)=\frac{z}{z(x+1)-1},
 \qquad
 x-Y(x)=\frac{zx^2-x+1}{z(x+1)-1}
\]
leads to the desired equations.
\end{proof}

The substitutions also have an analytic interpretation near the origin.
By the Marcus--Tardos theorem~\cite{MT}, $d_n^{(j)}\leq (M^{(j)})^n$ for some
$M^{(j)}$. Since every state has $k+\ell<n$, the trivariate series converges
absolutely when
\[
 M^{(j)}|z|\max(1,|x|,|y|)<1.
\]

Next, parameterise
\begin{equation}\label{eq:z-t-common}
 z=\frac{t}{(1+t)^2},
\end{equation}
using at the origin the inverse branch
\begin{equation}\label{eq:t-of-z}
 t=t(z)=\frac{1-2z-\sqrt{1-4z}}{2z}
 =z+2z^2+5z^3+14z^4+\cdots,
\end{equation}
and introduce
\begin{equation}\label{eq:w-change-common}
 w=w(x)=\frac{1+t-x}{1+t-tx},
 \qquad
 x=x(w)=\frac{(1+t)(1-w)}{1-tw}.
\end{equation}
A direct calculation gives
\begin{equation}\label{eq:common-conjugacy}
 w(Y(x))=m^2w(x).
\end{equation}
This is the common endpoint of the combinatorial part of the argument.
The remainder of the paper studies the two different $q$-difference
equations obtained from \eqref{eq:kernel-three} and
\eqref{eq:kernel-two} by the change of variables from $x$ to $w$
(here, $q=t^2$).

\section{The class $ \Av(4123,4231,4312) $}\label{sec:three-pattern}
\subsection{Solution of the $q$-difference equation}\label{sec:three-quotient}
We first apply the equations of the previous section to the calculation of $D^{(3)}$. Define
$\Phi_t(w)$ and $\widehat g^{(3)}(t)$ by
\begin{equation}\label{eq:three-Phi-def}
 G^{(3)}(z;x(w))=t^{-1}(1-tw)\Phi_t(w),
 \qquad
 \widehat g^{(3)}(t)=g^{(3)}(z(t)).
\end{equation}

\begin{proposition}\label{prop:three-q-equation}
  Equation \eqref{eq:kernel-three} is equivalent to
  the $q$-difference equation
\begin{align}
0={}&t^2\widehat g^{(3)}(t)(1-w)-t(1+t)(t-w)\notag\\
 &{}-(1-t)^2(1+t)w\,\Phi_t(w)\notag\\
 &{}-t(t-w)(1-t^2w)\Phi_t(t^2w).
\label{eq:three-q-equation}
\end{align}
\end{proposition}
\begin{proof}
Substitute \eqref{eq:z-t-common}, \eqref{eq:w-change-common}, and
\eqref{eq:three-Phi-def} into \eqref{eq:kernel-three}, and use
\eqref{eq:common-conjugacy}. Cancellation of the common rational factors
gives \eqref{eq:three-q-equation}.
\end{proof}

For coefficient extraction, replace $w$ by $t^{-2}w$ in
\eqref{eq:three-q-equation} and solve for $\Phi_t(w)$. This gives
\begin{equation}\label{eq:three-outward}
 \Phi_t(w)=\widehat g^{(3)}(t)\beta_t^{(3)}(w)+\gamma_t^{(3)}(w)
 +\alpha_t^{(3)}(w)\Phi_t(t^{-2}w),
\end{equation}
where
\begin{align*}
 \alpha_t^{(3)}(w)
 &=-\frac{(1-t)^2(1+t)w}{t(1-w)(t^3-w)},
\\
 \beta_t^{(3)}(w)
 &=\frac{t(t^2-w)}{(1-w)(t^3-w)},
\\
 \gamma_t^{(3)}(w)&=-\frac{1+t}{1-w}.
\end{align*}
Although multiplication of the argument by $t^{-2}$ moves it away from
the origin when $0<t<1$, the iteration is legitimate as an identity of
formal power series in $w$. Indeed,
\[
 \prod_{i=0}^{j-1}\alpha_t^{(3)}(t^{-2i}w)
 \in w^j\mathbb Q(t)[[w]],
\]
so after $N$ iterations the unexpanded remainder is divisible by $w^N$.
Thus the coefficient of $w^n$ has stabilised once $N>n$; the numerical
size of $t^{-2i}w$ is irrelevant to this coefficientwise argument. Define
the resulting formal series
\begin{align*}
 U_t(w)&=\sum_{j\geq0}\beta_t^{(3)}(t^{-2j}w)
 \prod_{i=0}^{j-1}\alpha_t^{(3)}(t^{-2i}w),
 \\
 V_t(w)&=\sum_{j\geq0}\gamma_t^{(3)}(t^{-2j}w)
 \prod_{i=0}^{j-1}\alpha_t^{(3)}(t^{-2i}w).
\end{align*}
For each coefficient of $w$, only finitely many summands contribute, and
iteration gives
\begin{equation}\label{eq:three-Phi-UV}
 \Phi_t(w)=\widehat g^{(3)}(t)U_t(w)+V_t(w).
\end{equation}
For $t\neq0$, define $R_k(t)$ and $S_k(t)$ by the scaled coefficient
expansions
\begin{align}
 U_t(w)&=\sum_{k\geq0}(-1)^k
 \frac{\bigl((1+t)(1-t)^2\bigr)^k}{t^{k(k+3)}}R_k(t)w^k,
 \notag\\
 V_t(w)&=\sum_{k\geq0}(-1)^k
 \frac{\bigl((1+t)(1-t)^2\bigr)^k}{t^{k(k+3)}}S_k(t)w^k.
 \label{eq:three-RS-coefficient-def}
\end{align}
The scaling is chosen so that the resulting recurrence approaches the
identity recurrence $X_k=X_{k-1}$ as $k\to\infty$.

\begin{proposition}\label{prop:three-scaled-recurrence}
The apparent negative powers of $t$ in
\eqref{eq:three-RS-coefficient-def} cancel. The initial values are
\begin{align}
 R_0(t)&=1,&
 R_1(t)&=\frac{1-2t+t^3-t^4}{(1+t)(1-t)^2},
 \label{eq:three-R-initial}\\
 S_0(t)&=-(1+t),&
 S_1(t)&=-\frac{1-t-t^2+t^3-t^4}{(1-t)^2}.
 \label{eq:three-S-initial}
\end{align}
For $X_k=R_k$ and for $X_k=S_k$,
\begin{equation}\label{eq:three-scaled-recurrence}
 X_k=\left(1-t^{2k-1}\frac{1-t+t^2}{(1-t)^2}\right)X_{k-1}
 -\frac{t^{4k-1}}{(1+t)^2(1-t)^4}X_{k-2},
 \qquad k\geq2.
\end{equation}
\end{proposition}
\begin{proof}
Separating the coefficient of $\widehat g^{(3)}(t)$ from the constant
term in \eqref{eq:three-outward} gives
\begin{align*}
 U_t(w)&=\beta_t^{(3)}(w)
 +\alpha_t^{(3)}(w)U_t(t^{-2}w),\\
 V_t(w)&=\gamma_t^{(3)}(w)
 +\alpha_t^{(3)}(w)V_t(t^{-2}w).
\end{align*}
Insert \eqref{eq:three-RS-coefficient-def} into these two equations.
Comparison at orders $1$ and $w$ gives
\eqref{eq:three-R-initial}--\eqref{eq:three-S-initial}; comparison at
order $w^k$, followed by cancellation of the common scaling factor,
gives \eqref{eq:three-scaled-recurrence}. The displayed formulas and
the recurrence also show that no negative powers of $t$ remain.
\end{proof}

\begin{proposition}\label{prop:three-wronskian}
For $k\geq1$, the discrete Wronskian
\[
  W_k(t):=R_k(t)S_{k-1}(t)-R_{k-1}(t)S_k(t)
\]
is given by
\begin{equation}
 W_k(t)=\frac{t^{2k^2+k-2}}
 {(1+t)^{2k-2}(1-t)^{4k-3}}.
\label{eq:three-wronskian}
\end{equation}
\end{proposition}
\begin{proof}
If the coefficient of $X_{k-2}$ in
\eqref{eq:three-scaled-recurrence} is $-b_k$, substitution gives
$W_k=b_kW_{k-1}$. The initial values give $W_1=t/(1-t)$; iteration gives
\eqref{eq:three-wronskian}.
\end{proof}

In particular, for $0<|t|<1$, the two solutions $(R_k)$ and $(S_k)$ of
the recurrence relation \eqref{eq:three-scaled-recurrence} are linearly independent.

\begin{proposition}\label{prop:three-RS-limits}
The sequences $(R_k)$ and $(S_k)$ converge locally uniformly in $|t|<1$.
Their limits
\[
 R(t)=\lim_{k\to\infty}R_k(t),
 \qquad
 S(t)=\lim_{k\to\infty}S_k(t)
\]
are analytic there, with $R(0)=1$, $S(0)=-1$, and
\[
 R_k-R=O(t^{2k+1}),
 \qquad
 S_k-S=O(t^{2k+1})
\]
coefficientwise.
\end{proposition}
\begin{proof}
Fix $0<r<1$. On $|t|\leq r$, write the two coefficients on the right of
\eqref{eq:three-scaled-recurrence} as $1+u_k(t)$ and $v_k(t)$. For
constants $C_r,C'_r$ depending only on $r$,
\[
 |u_k(t)|\leq C_rr^{2k-1},
 \qquad
 |v_k(t)|\leq C'_rr^{4k-1}.
\]
If $M_r$ is a uniform bound for $|X_k|+|X_{k-1}|$, the recurrence gives
the estimate
\[
 |X_k-X_{k-1}|
 \leq M_r\bigl(C_rr^{2k-1}+C'_rr^{4k-1}\bigr).
\]
Such an $M_r$ exists because iterating the preceding recurrence bounds
$|X_k|+|X_{k-1}|$ by the initial norm times a product of factors
$1+O(r^{2k-1})+O(r^{4k-1})$, and that product converges. Summing the
displayed estimate shows uniformly on $|t|\leq r$ that $(X_k)$ is Cauchy.
Local uniform convergence proves analyticity. Finally, the powers in the
recurrence show inductively that the Taylor coefficients stabilise
through degree $2k$, which gives the asserted coefficientwise estimates;
the values at zero follow from the initial conditions.
\end{proof}

\begin{theorem}\label{thm:three-quotient}
For $z=t/(1+t)^2$ with sufficiently small $|t|$,
\[
 g^{(3)}(z)=-\frac{S(t)}{R(t)},
\]
and consequently
\begin{equation}\label{eq:three-D-quotient}
 D^{(3)}(z)=1-z\frac{S(t)}{R(t)}.
\end{equation}
\end{theorem}
\begin{proof}
For sufficiently small $|z|$, the convergence estimate in
\S\ref{sec:common-kernel} shows that $G^{(3)}(z;x)$ is analytic for $x$
in a fixed neighbourhood of $1$. Since $x(0)=1+t$,
\eqref{eq:three-Phi-def} makes $\Phi_t$ analytic near $w=0$; its Taylor
coefficients therefore grow at most exponentially.

By \eqref{eq:three-Phi-UV} and
\eqref{eq:three-RS-coefficient-def}, the coefficient of $(-w)^k$ in
$\Phi_t(w)$ is
\begin{equation}\label{eq:three-Phi-coefficient}
 \frac{\bigl((1+t)(1-t)^2\bigr)^k}{t^{k(k+3)}}
 \bigl(\widehat g^{(3)}(t)R_k(t)+S_k(t)\bigr).
\end{equation}
If $\widehat g^{(3)}(t)R(t)+S(t)\neq0$, the final factor stays bounded away
from zero as $k\to\infty$. Since $0<|t|<1$, the $k$th root of the absolute
value of \eqref{eq:three-Phi-coefficient} then tends to infinity,
contradicting the positive radius of convergence of $\Phi_t$. Hence
$\widehat g^{(3)}(t)R(t)+S(t)=0$. Since $R(0)=1$, division is valid near
the origin, and \eqref{eq:Dj-from-gj} gives
\eqref{eq:three-D-quotient}.
\end{proof}

As a check,
\[
 -\frac{S(t)}{R(t)}
 =1+2t+2t^2+3t^3+5t^4+9t^5+17t^6+32t^7+\cdots,
\]
and substitution of $t=t(z)$ recovers the coefficients $d_n^{(3)}$ in the
introduction.

\subsection{Oscillation and non-cancellation}\label{sec:three-zeros}
We now prove that the denominator in \eqref{eq:three-D-quotient} has
infinitely many zeros in the unit disc and that none are cancelled by the
numerator. We use the standard notation
\[
 (a;t)_j=\prod_{i=0}^{j-1}(1-at^i),
 \qquad
 (a;t)_\infty=\prod_{i=0}^\infty(1-at^i).
\]

\begin{proposition}\label{prop:three-R-series}
For $|t|<1$,
\begin{equation}\label{eq:three-R-series}
 R(t)=\sum_{j\geq0}(-1)^jt^{j^2}
 \frac{1-(1+t-t^2)(1-t^j)}
 {(1+t)^j(1-t)^{2j-1}(t;t)_j}.
\end{equation}
The series converges locally uniformly in the open unit disc.
\end{proposition}
\begin{proof}
Let $T_j(t)$ denote the $j$th summand in
\eqref{eq:three-R-series}. Iteration of
\eqref{eq:three-scaled-recurrence} coefficientwise at $t=0$ gives
\begin{equation}\label{eq:three-Rk-congruence}
 R_k(t)=\sum_{j\geq0}(-1)^jt^{j^2}
 \frac{1-(1+t-t^2)(1-t^j)}
 {(1+t)^j(1-t)^{2j-1}(t;t)_j}+O(t^{2k+1}).
\end{equation}
Indeed, the initial cases follow from \eqref{eq:three-R-initial}; after
substitution of the induction hypothesis, the passage from the
$(j-1)$st term to the $j$th term uses the explicit identity
\[
 \frac{t^{(j-1)^2}t^{2j-1}}{(t;t)_{j-1}}
 =\frac{t^{j^2}(1-t^j)}{(t;t)_j}.
\]
Thus the terms containing $1-t^j$ shift the summation index by one, while
the two unshifted contributions combine, using
$1-(1+t-t^2)(1-t^j)$, into $T_j(t)$. Terms not yet determined after the
$k$th recurrence step have degree at least $2k+1$, which proves
\eqref{eq:three-Rk-congruence}. Proposition
\ref{prop:three-RS-limits} now proves the formal identity. On $|t|\leq r<1$
the products $(t;t)_j$ are bounded away from zero uniformly in $j$, while
$t^{j^2}$ dominates every exponential factor. The Weierstrass test gives
local uniform convergence.
\end{proof}

\begin{remark}
For $0<t<1$, let
\[
 A_t(x)=\sum_{j\geq0}\frac{t^{j^2}(-x)^j}{(t;t)_j}
\]
denote Ramanujan's entire function. Separating the two terms in the
numerator of \eqref{eq:three-R-series} gives the exact identity
\[
 R(t)=-t(1-t)^2
 A_t\!\left(\frac{1}{(1+t)(1-t)^2}\right)
 +(1-t)(1+t-t^2)
 A_t\!\left(\frac{t}{(1+t)(1-t)^2}\right).
\]
Thus, $R(t)$ is a linear combination of two values of Ramanujan's entire
function whose arguments tend to $+\infty$ as $t\uparrow1$.
Similarly, one can easily show that
\[
 S(t)=-(1+t)A_t\left(\frac{t^2}{(1+t)(1-t)^2}\right).
\]
\end{remark}

It is not immediately clear to this author how to use the remark above for our purposes,
and instead we calculate as follows.

Define
\[
 \varepsilon=-\log t,
 \qquad
 \nu=-\frac{\log\bigl((1-t)\sqrt{1+t}\bigr)}{\varepsilon},
\]
so that $t^\nu=(1-t)\sqrt{1+t}$,
and for real $u$,
\begin{equation}\label{eq:three-F-epsilon}
 F_\varepsilon(u)=e^{-\varepsilon(u-\nu)^2}
 \bigl((1+t-t^2)t^u-t(1-t)\bigr)(t^{u+1};t)_\infty.
\end{equation}
Note that as $t\uparrow1$,
\[
 \varepsilon\sim1-t,
 \qquad
 \nu\sim\frac{1}{1-t}\log\frac{1}{\sqrt2(1-t)}\longrightarrow\infty.
\]
so that the Gaussian part of $F_\varepsilon(t)$ has centre $\nu$, which goes to infinity,
and width of order $\varepsilon^{-1/2}\sim(1-t)^{-1/2}$.

Using the elementary identities
\begin{align*}
 (1+t)^j(1-t)^{2j-1}&=\frac{t^{2\nu j}}{1-t},\notag\\
 t^{j^2-2\nu j}
 &=e^{\varepsilon\nu^2}e^{-\varepsilon(j-\nu)^2},
\\
 \frac{(t;t)_\infty}{(t;t)_j}&=(t^{j+1};t)_\infty\notag
\end{align*}
and multiplying \eqref{eq:three-R-series} by $(t;t)_\infty$ leads to the formula
\begin{equation}\label{eq:three-R-bilateral}
 \frac{(t;t)_\infty R(t)}{(1-t)e^{\varepsilon\nu^2}}
 =\sum_{j\in\mathbb Z}(-1)^jF_\varepsilon(j).
\end{equation}
We have extended the summation to negative integers
because the last product in $F_\varepsilon(j)$, $j\in \ZZ_{<0}$, contains $1-t^0$ and hence vanishes.

We want to use Poisson summation to obtain the asymptotic behaviour of $R(t)$ as $t\uparrow1$.
With the Fourier-transform convention
$\widehat F(\xi)=\int_{-\infty}^\infty F(u)e^{-2\pi i\xi u}\,du$, one has
\begin{lemma}\label{lem:three-F-Fourier}
For $0<t<1$, the function $F_\varepsilon$ is Schwartz and
\begin{equation}\label{eq:three-F-Fourier}
 \widehat F_\varepsilon(\xi)
 =(1-t)\sqrt{\frac{\pi}{\varepsilon}}
 e^{-2\pi i\xi\nu-\pi^2\xi^2/\varepsilon}
 K_\varepsilon(e^{i\pi\xi}),
\end{equation}
where
\begin{align}
 K_\varepsilon(w)=\sum_{\ell\geq0}
 \frac{(-1)^\ell t^{\nu\ell}}{(t;t)_\ell}
 \bigl(&(1+t-t^2)\sqrt{1+t}\,
 t^{(\ell^2-1)/4}w^{\ell+1}\notag\\[-2mm]
 &{}-t^{1+\ell(\ell+2)/4}w^\ell\bigr).
\label{eq:three-K-epsilon}
\end{align}
Uniformly on $|w|=1$,
\begin{equation}\label{eq:three-K-limit}
 K_\varepsilon(w)\longrightarrow
 K_0(w)=(\sqrt2w-1)e^{-\sqrt2w}
 \qquad(t\uparrow1),
\end{equation}
and $K_\varepsilon$ is uniformly bounded there for $t$ sufficiently close
to $1$.
\end{lemma}
\begin{proof}
Euler's identity gives
\[
 (t^{u+1};t)_\infty
 =\sum_{\ell\geq0}
 \frac{(-1)^\ell t^{\ell(\ell+1)/2}t^{\ell u}}
 {(t;t)_\ell}.
\]
Insert this expansion into $F_\varepsilon$, write $u=\nu+s$, and use
$t^\nu=(1-t)\sqrt{1+t}$. For each $\ell$, the two terms in the factor
$(1+t-t^2)t^u-t(1-t)$ lead respectively to the Gaussian integrals with
$a=\ell+1$ and $a=\ell$ in
\begin{align}
 \int_{-\infty}^\infty
 e^{-\varepsilon s^2-a\varepsilon s-2\pi i\xi s}\,ds
 =\sqrt{\frac{\pi}{\varepsilon}}
 \exp\left(\frac{a^2\varepsilon}{4}+i\pi a\xi
 -\frac{\pi^2\xi^2}{\varepsilon}\right).
\label{eq:three-Gaussian-transform}
\end{align}
Collecting the common Gaussian factor gives
\eqref{eq:three-F-Fourier} and \eqref{eq:three-K-epsilon}.

For convergence, the arithmetic--geometric mean inequality gives
\[
 1-t^k\geq(1-t)kt^{(k-1)/2},
\]
and hence
\[
 (t;t)_\ell\geq(1-t)^\ell\ell!t^{\ell(\ell-1)/4}.
\]
Since $t^\nu=(1-t)\sqrt{1+t}$, the two series in
\eqref{eq:three-K-epsilon} are bounded by constant multiples of
$\sum (\sqrt{1+t})^\ell/\ell!$. This justifies termwise integration and
gives the uniform bound. For fixed $\ell$,
\[
 \frac{t^{\nu\ell}}{(t;t)_\ell}\longrightarrow
 \frac{(\sqrt2)^\ell}{\ell!};
\]
dominated convergence yields \eqref{eq:three-K-limit}.

Rapid decay of $F_\varepsilon$ and its derivatives as $u\to+\infty$ is
immediate from \eqref{eq:three-F-epsilon}. As $u\to-\infty$, split
finitely many factors from $(t^{u+1};t)_\infty$ until the remaining
argument lies in a fixed compact interval. The finite product grows at
most like $\exp(\varepsilon u^2/2+O_\varepsilon(|u|))$, which is dominated
by the Gaussian $e^{-\varepsilon(u-\nu)^2}$. The same argument after
differentiation proves the Schwartz assertion.
\end{proof}

The Gaussian factor in \eqref{eq:three-F-Fourier} suppresses all
but the lowest frequencies in the Poisson summation, leading to a simple oscillatory behaviour:
\begin{theorem}\label{thm:three-R-asymptotic}
Set
\[
 \delta=\sqrt2-\pi+\arctan\sqrt2.
\]
As $t\uparrow1$,
\begin{equation}
 R(t)=\frac{2\sqrt{3\pi}(1-t)^2}
 {(t;t)_\infty\sqrt\varepsilon}
 \exp\left(\varepsilon\nu^2-\frac{\pi^2}{4\varepsilon}\right)
 \bigl(\cos(\pi\nu+\delta)+o(1)\bigr).
\label{eq:three-R-asymptotic}
\end{equation}
The error term 
is uniform as $t\uparrow1$.
\end{theorem}
\begin{proof}
Poisson summation and Lemma~\ref{lem:three-F-Fourier} give
\[
 \sum_{j\in\mathbb Z}(-1)^jF_\varepsilon(j)
 =(1-t)\sqrt{\frac{\pi}{\varepsilon}}
 \sum_{n\in\mathbb Z}
 e^{-2\pi i(n-1/2)\nu-\pi^2(n-1/2)^2/\varepsilon}K_\varepsilon(e^{i\pi(n-1/2)}).
\]
The terms with $|n-1/2|\geq3/2$ are
$O((1-t)\varepsilon^{-1/2}e^{-9\pi^2/(4\varepsilon)})$. The two
remaining terms are complex conjugates, and
\[
 K_\varepsilon(-i)\longrightarrow(-1-i\sqrt2)e^{i\sqrt2}
 =\sqrt3e^{i\delta},
 \qquad
 K_\varepsilon(i)\longrightarrow\sqrt3e^{-i\delta}.
\]
Thus
\[
 \sum_{j\in\mathbb Z}(-1)^jF_\varepsilon(j)
 =2\sqrt{3\pi}(1-t)\varepsilon^{-1/2}
 e^{-\pi^2/(4\varepsilon)}
 \bigl(\cos(\pi\nu+\delta)+o(1)\bigr).
\]
Substitution into \eqref{eq:three-R-bilateral} proves the result.
\end{proof}

\begin{corollary}\label{cor:three-positive-zeros}
The function $R$ has infinitely many sign-changing zeros in $(0,1)$,
accumulating at $1$.
\end{corollary}
\begin{proof}
The prefactor in \eqref{eq:three-R-asymptotic} is positive, while
\begin{equation}\label{eq:three-nu-asymptotic}
 \nu(t)=\frac{\log(1/\varepsilon)-\frac12\log2}{\varepsilon}
 +\frac34+O(\varepsilon),
 \qquad \varepsilon\downarrow0.
\end{equation}
It tends to infinity and is eventually strictly increasing as a function
of $t$. For each sufficiently large integer $n$, choose the unique
$\widehat t_n\to1$ with $\pi\nu(\widehat t_n)+\delta=n\pi$. Theorem
\ref{thm:three-R-asymptotic} gives
$\operatorname{sign}R(\widehat t_n)=(-1)^n$.
The intermediate value theorem supplies a zero between consecutive phase
points.
\end{proof}

It remains to rule out cancellation. For fixed $t$ with $0<|t|<1$,
write \eqref{eq:three-scaled-recurrence} as
\begin{equation}\label{eq:three-ab-recurrence}
 X_k=a_kX_{k-1}-b_kX_{k-2},
\end{equation}
where
\[
 a_k=1-t^{2k-1}\frac{1-t+t^2}{(1-t)^2},
 \qquad
 b_k=\frac{t^{4k-1}}{(1+t)^2(1-t)^4}.
\]
Then $\sum_{k\geq2}(|a_k-1|+|b_k|)<\infty$.

For $2\leq s\leq N$, let
\[
 D_{s,N}=\det
 \begin{pmatrix}
 a_s&1&&&\\
 b_{s+1}&a_{s+1}&1&&\\
 &b_{s+2}&a_{s+2}&\ddots&\\
 &&\ddots&\ddots&1\\
 &&&b_N&a_N
 \end{pmatrix},
\]
with $D_{N+1,N}=1$ and $D_{N+2,N}=0$.

\begin{lemma}\label{lem:three-tail-continuant}
For every $s\geq2$, the limit $D_s=\lim_{N\to\infty}D_{s,N}$ exists,
and
\begin{equation}\label{eq:three-tail-recurrence}
 D_s=a_sD_{s+1}-b_{s+1}D_{s+2},
 \qquad D_s\longrightarrow1.
\end{equation}
For any solution of \eqref{eq:three-ab-recurrence},
\begin{equation}\label{eq:three-limit-functional}
 \lim_{k\to\infty}X_k=D_2X_1-b_2D_3X_0.
\end{equation}
\end{lemma}
\begin{proof}
Put $u_k=a_k-1$. In the determinant expansion of $D_{s,N}$, each
nonidentity permutation is a collection of disjoint transpositions of
adjacent indices. Expanding each diagonal entry as $1+u_k$ and dropping
the disjointness restriction gives
\[
 |D_{s,N}-1|\leq
 \exp\left(\sum_{k=s}^N|u_k|+\sum_{k=s+1}^N|b_k|\right)-1.
\]
The associated infinite expansion is absolutely convergent, proving the
existence of $D_s$ and the limit $D_s\to1$. Expansion in the first row
gives the recurrence in \eqref{eq:three-tail-recurrence}. Finally, direct
induction gives
$X_N=D_{2,N}X_1-b_2D_{3,N}X_0$; let $N\to\infty$.
\end{proof}

\begin{theorem}\label{thm:three-noncancellation}
The analytic functions $R$ and $S$ have no common zero in $|t|<1$.
\end{theorem}
\begin{proof}
At $t=0$ this follows from $R(0)=1$ and $S(0)=-1$. Fix
$0<|t|<1$. Applying \eqref{eq:three-limit-functional} to $(R_k)$ and
$(S_k)$ gives
\[
 R=D_2R_1-b_2D_3R_0,
 \qquad
 S=D_2S_1-b_2D_3S_0.
\]
If $R=S=0$, Proposition~\ref{prop:three-wronskian} at $k=1$ gives
\[
 R_1S_0-R_0S_1=\frac{t}{1-t}\neq0.
\]
Thus the preceding two equations force $D_2=D_3=0$. Since every $b_k$ is
nonzero, \eqref{eq:three-tail-recurrence} then forces $D_s=0$ for every
$s\geq2$, contradicting $D_s\to1$.
\end{proof}

\subsection{Singularities and asymptotics}\label{sec:three-asymptotics}
The map
$
 z=z(t)=\frac{t}{(1+t)^2}
$
is injective in $|t|<1$, since the two preimages of a generic value are
reciprocal, and
$
 z'(t)=\frac{1-t}{(1+t)^3}
$
does not vanish there.

\begin{proposition}\label{prop:three-singularity}
If $|t_0|<1$, $R(t_0)=0$, and $S(t_0)\neq0$, then $D^{(3)}$ has a
singularity at $z_0=z(t_0)$. If $t_0$ is a simple zero, the
singularity is a simple pole.
\end{proposition}
\begin{proof}
Equation \eqref{eq:three-D-quotient} gives a meromorphic continuation of
$D^{(3)}(z(t))$ along every path in the unit disc avoiding the isolated zeros
of $R$. Approaching $t_0$ along such a path, the quotient $S/R$ is
singular. Since $z'(t_0)\neq0$, the change of variables is locally
biholomorphic and preserves the order of the singularity.
\end{proof}

\begin{theorem}\label{thm:three-non-D-finite}
The generating function $D^{(3)}$ has infinitely many finite singularities.
It is not D-finite, and $(d_n^{(3)})_{n\geq0}$ is not P-recursive.
\end{theorem}
\begin{proof}
Corollary~\ref{cor:three-positive-zeros}, Theorem
\ref{thm:three-noncancellation}, and Proposition
\ref{prop:three-singularity} provide infinitely many distinct finite
singularities. A solution of a linear differential equation with
polynomial coefficients can have finite singularities only at the finitely
many zeros of its leading coefficient. Thus $D^{(3)}$ is not D-finite. The
equivalence between D-finiteness of an ordinary generating function and
P-recursiveness of its coefficients gives the last assertion
\cite{StanleyDF}.
\end{proof}

The first of these singularities controls the growth of the coefficients.
Recall that the direct sum $\alpha\oplus\beta$ places $\alpha$ below and
to the left of $\beta$.
The quotient formula locates the poles, but does not by itself
show that the first positive pole is simple or that it is the only
singularity on its circle. To prove those facts, we use the direct-sum
decomposition to introduce a generating function with nonnegative
coefficients.

\begin{lemma}\label{lem:three-sum-decomposition}
The class $\Dthree$ is closed under direct sums. If $i_n$ is the number
of its sum-indecomposable members of size $n$ and
$I(z)=\sum_{n\geq1}i_nz^n$, then
\begin{equation}\label{eq:three-D-sequence}
 D^{(3)}(z)=\frac1{1-I(z)}.
\end{equation}
\end{lemma}
\begin{proof}
First we prove sum closure. Each basis element of $\Dthree$ begins with
its largest entry and is therefore sum indecomposable: in a nontrivial
direct sum, the largest entry belongs to the second block. Now let
$\alpha,\beta\in\Dthree$. Any pattern occurrence in
$\alpha\oplus\beta$ that uses entries from both blocks is itself a
nontrivial direct sum, because its entries from $\alpha$ all occur before
and below those from $\beta$. Such an occurrence cannot be one of the
three sum-indecomposable basis elements. An occurrence contained in one
block is also impossible, since both blocks avoid the basis. Hence
$\alpha\oplus\beta\in\Dthree$.

Every nonempty permutation has a unique factorization
\[
 \pi=\pi_1\oplus\cdots\oplus\pi_k
\]
into sum-indecomposable permutations. Indeed, the first component ends
at the least index $s$ for which the first $s$ entries have values
$\{1,\ldots,s\}$, and the construction then continues recursively.
Because a permutation class is closed under taking patterns, all the
components of a member of $\Dthree$ lie in $\Dthree$; conversely, their
direct sum lies in $\Dthree$ by the closure just proved. Therefore a
member of $\Dthree$ is uniquely a sequence of sum-indecomposable members,
and the sequence construction gives
\[
 D^{(3)}(z)=1+I(z)+I(z)^2+\cdots=\frac1{1-I(z)}.
\]
\end{proof}

One can easily compute the first few orders of $I(z)$:
\[
I(z)=z+z^2+3z^3+10z^4+\cdots
\]

Note that poles of $D^{(3)}(z)$ correspond to solutions of $I(z)=1$.

\begin{theorem}\label{thm:three-dominant-pole}
Let
\begin{equation}\label{eq:three-tstar-rho}
 t^{(3)}=\min\{t\in(0,1):R(t)=0\},
 \qquad
 \rho^{(3)}=z(t^{(3)}).
\end{equation}
Then $\rho^{(3)}$ is the radius of convergence of $D^{(3)}$, its unique
singularity on $|z|=\rho^{(3)}$, and a simple pole. For some
$r^{(3)}\in(\rho^{(3)},1/4)$,
\begin{equation}\label{eq:three-dn-asymptotic}
 d_n^{(3)}=C^{(3)}(\rho^{(3)})^{-n}+O((r^{(3)})^{-n}),
 \qquad
 C^{(3)}=\frac1{\rho^{(3)}I'(\rho^{(3)})}
 =\frac{S(t^{(3)})z'(t^{(3)})}{R'(t^{(3)})}>0.
\end{equation}
In particular, $t^{(3)}$ is a simple zero of $R$.
\end{theorem}
\begin{proof}
The inverse $t(z)$ in \eqref{eq:t-of-z} is analytic in $|z|<1/4$ and
takes values in $|t|<1$. Indeed, if
$|t|=1$, then $|t/(1+t)^2|=|1+t|^{-2}\geq1/4$. The quotient formula and
Theorem~\ref{thm:three-noncancellation} therefore make $D^{(3)}$ meromorphic
in $|z|<1/4$, with poles precisely at the images of zeros of $R$.

Positive zeros exist and are isolated, so $t^{(3)}$ is well defined and
$\rho^{(3)}$ is a singularity. If the radius of convergence were smaller
than $\rho^{(3)}$, the Vivanti--Pringsheim theorem would produce a positive singularity
whose inverse image lies in $(0,t^{(3)})$, contradicting the quotient formula.
Thus the radius is $\rho^{(3)}$.

For real $s\uparrow\rho^{(3)}$, positivity of the coefficients and the pole at
$\rho^{(3)}$ give $D^{(3)}(s)\to+\infty$. Equation
\eqref{eq:three-D-sequence} then yields
\begin{equation}\label{eq:three-I-rho}
 I(\rho^{(3)})=\sum_{n\geq1}i_n(\rho^{(3)})^n=1.
\end{equation}
Suppose $\zeta$ is a pole with $|\zeta|=\rho^{(3)}$. Letting $r\uparrow1$ in
$I(r\zeta)=1-D^{(3)}(r\zeta)^{-1}$ gives $I(\zeta)=1$, since the series for
$I$ converges absolutely on the circle. Hence
\[
 1=|I(\zeta)|\leq\sum_{n\geq1}i_n|\zeta|^n=I(\rho^{(3)})=1.
\]
Equality in the triangle inequality, together with say $i_1=1$, forces
$\zeta=\rho^{(3)}$. Thus the dominant singularity is unique.

Near $\rho^{(3)}$, the function $\widetilde I=1-1/D^{(3)}$ is analytic and agrees
with $I$ on the interval to the left. Therefore
\[
 \widetilde I'(\rho^{(3)})=\sum_{n\geq1}ni_n(\rho^{(3)})^{n-1}\geq i_1=1.
\]
It follows from $D^{(3)}=(1-\widetilde I)^{-1}$ that the pole is simple and
has coefficient $1/(\rho^{(3)}I'(\rho^{(3)}))$. Comparing with
\eqref{eq:three-D-quotient} gives the second expression for $C^{(3)}$ and
shows that $t^{(3)}$ is a simple zero. Finally, subtract the principal part
and apply Cauchy's estimate in a slightly larger disc containing no other
pole.
\end{proof}

Numerical evaluation of the convergent recurrences gives
\begin{align}
 t^{(3)}&=0.5108368593557502883\ldots,\notag\\
 \rho^{(3)}&=0.2237933013470754447\ldots,\notag\\
 (\rho^{(3)})^{-1}&=4.468408991603930817\ldots,\notag\\
 C^{(3)}&=0.018967232507198778\ldots.
\label{eq:three-numerics}
\end{align}
In particular,
\begin{equation}\label{eq:three-root-ratio-limits}
 \lim_{n\to\infty}(d_n^{(3)})^{1/n}
 =\lim_{n\to\infty}\frac{d_{n+1}^{(3)}}{d_n^{(3)}}
 =(\rho^{(3)})^{-1}.
\end{equation}

\begin{remark}
In principle, one can improve on Theorem~\ref{thm:three-dominant-pole} by keeping track of the contributions of
other poles in the region $|z|<1/4$. 
Choosing one zero $\widetilde t_n$ between $\widehat t_n$ and $\widehat t_{n+1}$, in the notations of the
proof of Corollary~\ref{cor:three-positive-zeros}, we get a sequence of positive poles $\widetilde z_n\uparrow1/4$ such
that
\[
 \frac14-\widetilde z_n
 \sim\frac{(\log n)^2}{16n^2},
\]
(In fact, asymptotically, one can show that there is only one zero per interval, though we skip the proof.)
However, numerical investigation shows that there are other, complex roots of $R(t)$, even in the region corresponding to $|z|<1/4$.
\end{remark}

Numerically, the first positive zeros of $R$ and their images are
\[
\begin{array}{c|c|c}
j&t_j&z_j=t_j/(1+t_j)^2\\ \hline
1&0.5108368593557503&0.2237933013470754\\
2&0.6413539750612613&0.2380637598080282\\
3&0.7059827514180680&0.2425743178578199\\
4&0.7467481304475681&0.2447448543821405\\
5&0.7755166604126627&0.2460036940071294.
\end{array}
\]
The five values $z_j$ agree, to all digits displayed here, with the first
five positive singularities reported experimentally in Pantone's 2017
slides~\cite{PantoneSlides}.
The results above give a rigorous proof of the
numerical singularity pattern observed there,
in particular, of the accumulation of singularities at $1/4$.

\subsection{Probabilistic consequences}
\label{sec:three-probabilistic}
The sum decomposition leads to a nice probabilistic byproduct. Let $K_n$
be the number of sum-indecomposable components of a uniformly random
member of $\Dthree\cap \Sym_n$, and let $L_n$ be the size of its first
component.

\begin{corollary}\label{cor:three-components}
Define
\begin{align*}
 p_k&=i_k(\rho^{(3)})^k,\\
 \mu_\oplus&=\sum_{k\geq1}kp_k
 =\rho^{(3)}I'(\rho^{(3)})=(C^{(3)})^{-1},
\\
  \sigma_\oplus^2&=\sum_{k\geq1}(k-\mu_\oplus)^2p_k
                   \\
 &=(\rho^{(3)})^2I''(\rho^{(3)})
 +\rho^{(3)}I'(\rho^{(3)})-\mu_\oplus^2.
\end{align*}
Then
\begin{align*}
 \mathbb E K_n&=C^{(3)}n+O(1),\\
 \operatorname{Var}K_n&=
 \frac{\sigma_\oplus^2}{\mu_\oplus^3}n+O(1),
\end{align*}
and the centred, normalised variables converge to a standard Gaussian.
For every fixed $k\geq1$,
\[
 \mathbb P(L_n=k)\longrightarrow p_k.
\]
Numerically,
\[
 \mu_\oplus=52.7225044360300\ldots,
 \qquad
 \frac{\sigma_\oplus^2}{\mu_\oplus^3}
 =0.0394372750310189\ldots.
\]
\end{corollary}
\begin{proof}
Equation \eqref{eq:three-I-rho} makes $(p_k)$ a probability distribution.
It is nondegenerate: both $1$ and $21$ are sum indecomposable, so
$i_1=i_2=1$ and hence $p_1,p_2>0$. In particular,
$\sigma_\oplus^2>0$.
Moreover, $I=1-1/D^{(3)}$ is analytic in a disc slightly larger than
$\rho^{(3)}$: $D^{(3)}$ has no zero on $|z|\leq\rho^{(3)}$, by
\eqref{eq:three-D-sequence} in the open disc and by absolute convergence
on its boundary. Marking the number of components gives
\[
 D^{(3)}(z,u)=\frac1{1-uI(z)}.
\]
The meromorphic sequence schema and quasi-powers theorem
\cite[Chapter~IX]{FS} yield the moment estimates and central limit theorem.
Finally,
\[
 \mathbb P(L_n=k)
 =\frac{i_kd_{n-k}^{(3)}}{d_n^{(3)}}\longrightarrow i_k(\rho^{(3)})^k
\]
by Theorem~\ref{thm:three-dominant-pole}.
\end{proof}




\section{The class $ \Av(4123,4312) $}\label{sec:two-pattern}
\subsection{Solution of the $q$-difference equation}\label{sec:two-product}
We return to the common kernel equation \eqref{eq:kernel-two}. This time
define
\begin{equation}\label{eq:two-Psi-def}
 G^{(2)}(z;x(w))=(1-tw)\Psi_t(w)
\end{equation}
and introduce
\begin{align*}
 L_t(w)&=(1-tw)(1-t^2w),
\\
 H_t(w)&=1+\frac{(1+t)(1-3t+t^2)}t\,w+tw^2,
\\
 K_t(w)&=\frac{(1-t^2)^2w}{t(1-t^2w)}.
\end{align*}

\begin{proposition}\label{prop:two-q-equation}
Equation \eqref{eq:kernel-two} is equivalent to the $q$-difference equation
\begin{equation}\label{eq:two-q-equation}
 L_t(w)\Psi_t(t^2w)=H_t(w)\Psi_t(w)-K_t(w).
\end{equation}
\end{proposition}
\begin{proof}
Substitute \eqref{eq:z-t-common}, \eqref{eq:w-change-common}, and
\eqref{eq:two-Psi-def} into \eqref{eq:kernel-two}, then use
\eqref{eq:common-conjugacy} and cancel common rational factors.
\end{proof}

Away from the zeros of $H_t$, write \eqref{eq:two-q-equation} as
\begin{equation}\label{eq:two-inward}
 \Psi_t(w)=\alpha_t(w)\Psi_t(t^2w)+\beta_t(w),
 \qquad
 \alpha_t(w)=\frac{L_t(w)}{H_t(w)},
 \quad
 \beta_t(w)=\frac{K_t(w)}{H_t(w)}.
\end{equation}
Because $|t|<1$, this recurrence iterates toward the origin. Define
\begin{align}
 P_t(w)&=\prod_{i\geq0}\alpha_t(t^{2i}w),
 \label{eq:two-P-def}\\
 Q_t(w)&=\sum_{j\geq0}\beta_t(t^{2j}w)
 \prod_{i=0}^{j-1}\alpha_t(t^{2i}w).
 \label{eq:two-Q-def}
\end{align}
Let
\begin{equation}\label{eq:two-Sigma-def}
 \Sigma_t=\{w:H_t(w)=0\}\cup\{t^{-2}\},
 \qquad
 \Omega_t=\{w:t^{2i}w\notin\Sigma_t\text{ for every }i\geq0\}.
\end{equation}
Thus $\Sigma_t$ contains every pole of $\alpha_t$ and $\beta_t$.
The points excluded from $\Omega_t$ lie on the backward dilation orbits
of $\Sigma_t$: a starting point $w=t^{-2i}a$ is excluded because its
inward trajectory $w,t^2w,t^4w,\ldots$ reaches $a\in\Sigma_t$ after
$i$ steps.

\begin{lemma}\label{lem:two-PQ}
Fix $t$ with $0<|t|<1$. On every compact subset of $\Omega_t$, the
product and sum in \eqref{eq:two-P-def}--\eqref{eq:two-Q-def} converge
normally. Every solution of \eqref{eq:two-inward} on $\Omega_t$ that is
analytic at $w=0$ has the form
\begin{equation}\label{eq:two-Psi-PQ}
 \Psi_t(w)=c(t)P_t(w)+Q_t(w),
 \qquad c(t)=\Psi_t(0).
\end{equation}
\end{lemma}
\begin{proof}
At the origin,
\[
 \alpha_t(w)=1+O(w),
 \qquad
 \beta_t(w)=O(w).
\]
If $E\subset\Omega_t$ is compact, the avoidance condition in
\eqref{eq:two-Sigma-def} makes all the finitely many initial factors
holomorphic on a neighbourhood of $E$. For all sufficiently large $i$,
$t^{2i}E$ lies in a fixed neighbourhood of the origin, and hence, uniformly
for $w\in E$,
\[
 \alpha_t(t^{2i}w)=1+O(|t|^{2i}),
 \qquad
 \beta_t(t^{2i}w)=O(|t|^{2i}).
\]
The summability of these bounds gives normal convergence of the product
and sum.

Iterating \eqref{eq:two-inward} $N$ times gives
\begin{align*}
 \Psi_t(w)={}&
 \left(\prod_{i=0}^{N}\alpha_t(t^{2i}w)\right)
 \Psi_t(t^{2N+2}w)\\
 &{}+\sum_{j=0}^{N}\beta_t(t^{2j}w)
 \prod_{i=0}^{j-1}\alpha_t(t^{2i}w).
\end{align*}
Since $t^{2N+2}w\to0$, analyticity at the origin gives
$\Psi_t(t^{2N+2}w)\to c(t)$. Letting $N\to\infty$ proves
\eqref{eq:two-Psi-PQ}.
\end{proof}

Multiplication of $H_t(w)$ by $t$ removes its apparent pole at $t=0$. The
polynomial
\[
 \widetilde H_t(w)=t+(1+t)(1-3t+t^2)w+t^2w^2
\]
has a unique zero $r(t)$ analytic near the origin with $r(0)=0$:
\begin{equation}\label{eq:two-r-def}
 r(t)=\frac{-(1+t)(1-3t+t^2)+(1-t)\sqrt{\Delta(t)}}{2t^2},
\end{equation}
where
\[
 \Delta(t)=t^4-2t^3-5t^2-2t+1,
 \qquad \sqrt{\Delta(t)}=1+O(t).
\]
Thus $r(t)=-t-2t^2-6t^3-16t^4-46t^5-\cdots$.

The point $r(t)$ does not belong to $\Omega_t$, because it is a zero of
$H_t$. We therefore use the undivided equation
\eqref{eq:two-q-equation}, rather than \eqref{eq:two-inward}, at this
point. Analyticity of the combinatorial solution then forces
\begin{equation}\label{eq:two-root-cancellation}
 \Psi_t(t^2r(t))=\eta_t(r(t)),
 \qquad
 \eta_t(r)=-\frac{(1-t^2)^2r}
 {t(1-tr)(1-t^2r)^2}.
\end{equation}
The other zero of $H_t$ is
$s(t)=1/(tr(t))=-t^{-2}+O(t^{-1})$, while
$t^{2i+2}r(t)=O(t^{2i+3})$. Thus, for all sufficiently small nonzero
$t$, these inward iterates avoid $r(t)$, $s(t)$, and $t^{-2}$;
they also avoid the zeros $t^{-1}$ and $t^{-2}$ of $L_t$. It follows
that $t^2r(t)\in\Omega_t$ and $P_t(t^2r(t))\neq0$.
Lemma~\ref{lem:two-PQ} applied at $t^2r(t)$ therefore determines the
free constant:
\begin{equation}\label{eq:two-c-formula}
 c(t)=\frac{\eta_t(r(t))-Q_t(t^2r(t))}{P_t(t^2r(t))}.
\end{equation}

\begin{theorem}\label{thm:two-exact}
For $z=t/(1+t)^2$ with sufficiently small $|t|$, define $r,P_t,Q_t,c$
by \eqref{eq:two-r-def}, \eqref{eq:two-P-def}, \eqref{eq:two-Q-def}, and
\eqref{eq:two-c-formula}. Then
\begin{align}
 g^{(2)}(z)&=(1-t^2)\bigl(c(t)P_t(t)+Q_t(t)\bigr),
 \label{eq:two-g-exact}\\
 D^{(2)}(z)&=1+t\frac{1-t}{1+t}\bigl(c(t)P_t(t)+Q_t(t)\bigr),
 \label{eq:two-D-exact}
\end{align}
\end{theorem}
\begin{proof}
For small $t$, the point $x(r(t))$ lies in a fixed neighbourhood of $1$,
where the combinatorial series $G^{(2)}(z;x)$ converges. Hence $\Psi_t$ is
analytic at $r(t)$, and substitution into \eqref{eq:two-q-equation} gives
\eqref{eq:two-root-cancellation}. The orbit check preceding the theorem
allows Lemma~\ref{lem:two-PQ} to be applied at $t^2r(t)$ and gives
\eqref{eq:two-c-formula}. The forward iterates of $t$ are
$t^{2i+1}$; for small $t$ they also avoid $\Sigma_t$, so the lemma
at $w=t$ gives $\Psi_t(t)=c(t)P_t(t)+Q_t(t)$. Finally $x(t)=1$, so
$g^{(2)}(z)=G^{(2)}(z;1)=(1-t^2)\Psi_t(t)$. This proves
\eqref{eq:two-g-exact}, and \eqref{eq:Dj-from-gj} proves
\eqref{eq:two-D-exact}.
\end{proof}

Expansion gives
\[
 D^{(2)}(z)=1+z+2z^2+6z^3+22z^4+89z^5+382z^6+1711z^7+\cdots,
\]
in agreement with the succession rule.

\subsection{The root collision}\label{sec:two-collision}
The discriminant of $H_t$ as a quadratic polynomial in $w$ is
\begin{equation}\label{eq:two-H-discriminant}
 \operatorname{disc}_wH_t=\frac{(1-t)^2}{t^2}\Delta(t).
\end{equation}
Moreover,
\begin{equation}\label{eq:two-Delta-factor}
 \Delta(t)=t^2\left((2+t+t^{-1})^2
 -6(2+t+t^{-1})+1\right).
\end{equation}
The first collision on the positive $t$-axis is therefore at
\begin{equation}\label{eq:two-tc}
 t^{(2)}=\frac12+\sqrt2-\frac12\sqrt{5+4\sqrt2}
 =0.281971680061\ldots,
\end{equation}
and
\begin{equation}\label{eq:two-zc-growth}
 z^{(2)}:=z(t^{(2)})=3-2\sqrt2,
 \qquad
 (z^{(2)})^{-1}=3+2\sqrt2.
\end{equation}

\begin{proposition}\label{prop:two-positive-continuation}
The expression in Theorem~\ref{thm:two-exact} has an analytic continuation 
along $0<t<t^{(2)}$. Consequently, the radius of convergence of $D^{(2)}$ is at
least $z^{(2)}$.
\end{proposition}
\begin{proof}
For $0<t<t^{(2)}$, the coefficient of $w$ in $H_t$ is positive and its two
zeros are negative. The root $r(t)$ is the one nearer the origin. Thus
the arguments $t^{2i+2}r(t)$ in $P_t(t^2r(t))$ and $Q_t(t^2r(t))$ lie
strictly between $r(t)$ and $0$ and avoid both zeros; all $t^{2i+1}$ are
positive and avoid them as well. None of these arguments can equal
$t^{-2}>1$, the remaining point of $\Sigma_t$. The products and sums are therefore
analytic near each point of the interval. Every factor of
$P_t(t^2r(t))$ is positive, and the product is nonzero because its factors
differ from $1$ by a summable sequence. Hence \eqref{eq:two-c-formula}
has no singularity on the interval.

The map $t\mapsto z(t)$ is increasing on $(0,1)$. The second part of the Proposition
follows from the Vivanti--Pringsheim theorem.
\end{proof}

To test whether the collision is visible in $D^{(2)}$, temporarily regard $t$
and $r$ as independent and define
\begin{align}
 c(t,r)&=\frac{\eta_t(r)-Q_t(t^2r)}{P_t(t^2r)},
 \label{eq:two-ctr}\\
 \widehat g^{(2)}(t,r)&=(1-t^2)
 \bigl(c(t,r)P_t(t)+Q_t(t)\bigr).
 \label{eq:two-ghat}
\end{align}
At the collision,
\[
 r^{(2)}=-\frac{(1+t^{(2)})(1-3t^{(2)}+(t^{(2)})^2)}{2(t^{(2)})^2}
 =-1.8832035059135\ldots.
\]
All products and sums in \eqref{eq:two-ghat} are regular at
$(t^{(2)},r^{(2)})$, because $(t^{(2)})^2r^{(2)}$ lies strictly between
$r^{(2)}$ and $0$.

\begin{proposition}\label{prop:two-transversality}
At the collision point,
\begin{equation}\label{eq:two-nondegeneracy}
 \partial_r\widehat g^{(2)}(t^{(2)},r^{(2)})<0.
\end{equation}
Numerically,
\begin{equation}\label{eq:two-ghat-r-numeric}
 \partial_r\widehat g^{(2)}(t^{(2)},r^{(2)})
 =-0.2401712129238621458043835\ldots.
\end{equation}
\end{proposition}
\begin{proof}
Put $u=\sqrt{t^{(2)}}$ and $w_0=-u^3$, and, at $t=t^{(2)}$, write
$P=P_t$, $Q=Q_t$, $c=c(t,r^{(2)})$, and $\psi=cP+Q$. Since
$H_t(w)=(1+uw)^2$ and $r^{(2)}=-u^{-1}$, differentiation of the root
cancellation identity gives
\begin{align}
 P(w_0)\,\partial_rc(t,r^{(2)})
 &=\eta_t'(r^{(2)})-u^4\psi'(w_0),\notag\\
 \partial_r\widehat g^{(2)}(t,r^{(2)})
 &=(1-u^4)P(t)\,\partial_rc(t,r^{(2)}).
\label{eq:two-transversality-reduction}
\end{align}
The factors $P(w_0)$ and $(1-u^4)P(t)$ multiplying $\partial_rc$ are positive,
so it remains only to certify the sign of the first right-hand side.

The number $u$ is the root near $0.531$ of
$u^8-2u^6-5u^4-2u^2+1$, and exact rational evaluation isolates it in
$
0.53101005645956<u<0.53101005645958
$.
The recurrence coefficients of \eqref{eq:two-inward} are
\[
\alpha_{t^{(2)}}(s)=\frac{(1-u^2s)(1-u^4s)}{(1+us)^2},\qquad
\beta_{t^{(2)}}(s)=\frac{(1-u^4)^2s}{u^2(1-u^4s)(1+us)^2}.
\]
These are the factors from which $P$ and $Q$ are built in
\eqref{eq:two-P-def}--\eqref{eq:two-Q-def}; bounds on them and their
derivatives therefore control the truncation errors in
\[
 P(w_0),\qquad P'(w_0),\qquad Q(w_0),\qquad Q'(w_0).
\]
On $w_0\leq s\leq0$, direct interval estimates give
\begin{align*}
 u^4&<0.081,& |w_0|&<0.151,& |(\log\alpha)'|&<1.6,\\
 |\beta(s)|&<3.55|s|,& 0&<\beta'(s)<4.3.
\end{align*}
Consequently, replacing the four quantities above by their truncations
before index $8$ incurs an error below $1.5\cdot10^{-8}$. Exact rational
interval arithmetic on the first eight terms (all rational functions of
$u$) then gives
$1.26785<P(w_0)<1.26786$ and
$-0.50305<\eta_t'(r^{(2)})-u^4\psi'(w_0)<-0.50301$.
The first line of \eqref{eq:two-transversality-reduction} now gives
$\partial_rc<0$, and the second proves \eqref{eq:two-nondegeneracy}.
Direct evaluation gives \eqref{eq:two-ghat-r-numeric}.
\end{proof}

\begin{theorem}\label{thm:two-growth}
The point $z^{(2)}=3-2\sqrt2$ is a square-root singularity of $D^{(2)}$, and
\[
 \lim_{n\to\infty}(d_n^{(2)})^{1/n}=3+2\sqrt2.
\]
\end{theorem}
\begin{proof}
As $t\uparrow t^{(2)}$, one has the expansion
\begin{equation}\label{eq:two-r-local}
 r(t)=r^{(2)}+\frac{1-t^{(2)}}{2(t^{(2)})^2}
 \sqrt{-t^{(2)}\Delta'(t^{(2)})}
 \sqrt{1-\frac t{t^{(2)}}}
 +O\left(1-\frac t{t^{(2)}}\right).
\end{equation}
Theorem~\ref{thm:two-exact} says
$g^{(2)}(z(t))=\widehat g^{(2)}(t,r(t))$. Proposition
\ref{prop:two-transversality} therefore gives a nonzero square-root term.
Since $z'(t^{(2)})\neq0$, the same holds in the local variable
$1-z/z^{(2)}$. Thus $D^{(2)}$ is singular at $z^{(2)}$, while Proposition
\ref{prop:two-positive-continuation} rules out a smaller positive
singularity. The radius is $z^{(2)}$. Existence of the growth-rate limit also
follows directly from sum closure, since both basis permutations are sum
indecomposable; alternatively it follows from the coefficient asymptotic
proved below.
\end{proof}

\subsection{Asymptotics}\label{sec:two-asymptotics}
It remains to exclude other singularities on the dominant circle. Write
\[
 b(t)=(1+t)(1-3t+t^2)=1-2t-2t^2+t^3
\]
and define
\[
 v(t)=-tr(t).
\]
The equation $\widetilde H_t(r(t))=0$ becomes
\begin{equation}\label{eq:two-v-equation}
 tv(t)^2-b(t)v(t)+t^2=0,
 \qquad
 v(t)=\frac{t^2+tv(t)^2}{b(t)}.
\end{equation}
Thus $v$ is the branch of this algebraic equation that is analytic at the
origin, characterised by $v(t)=t^2+O(t^3)$.

\begin{lemma}\label{lem:two-v-bound}
The Taylor series of $v$ at the origin has nonnegative coefficients. For
$|t|<t^{(2)}$,
\begin{equation}\label{eq:two-v-bound}
 |v(t)|\leq v(|t|)<v(t^{(2)})=\sqrt{t^{(2)}}<1.
\end{equation}
\end{lemma}
\begin{proof}
Write $1/b(t)=\sum_{n\geq0}e_nt^n$. Its coefficients begin $1,2,6$ and
satisfy $e_n=2e_{n-1}+2e_{n-2}-e_{n-3}$. They are positive and strictly
increasing: once $e_{n-1}>e_{n-3}$, the recurrence gives
$e_n>e_{n-1}+2e_{n-2}>e_{n-1}$. Iteration of the second equation in
\eqref{eq:two-v-equation}, starting with $v=0$, proves coefficientwise
nonnegativity.

The identity \eqref{eq:two-Delta-factor} shows that the four zeros of
$\Delta$ have moduli $t^{(2)},1,1,(t^{(2)})^{-1}$. Hence $v$ is analytic for
$|t|<t^{(2)}$. At the positive boundary point the double-root factorization
gives $v(t^{(2)})=\sqrt{t^{(2)}}$. The triangle inequality for its nonnegative
series proves \eqref{eq:two-v-bound}.
\end{proof}

The two roots of $H_t$ are
\[
 r(t)=-\frac{v(t)}t,
 \qquad
 s(t)=-\frac1{v(t)},
\]
and hence
\[
 H_t(w)=\left(1+\frac t{v(t)}w\right)(1+v(t)w).
\]

\begin{proposition}\label{prop:two-t-disc}
The right-hand side of \eqref{eq:two-g-exact}, initially defined near
$t=0$, is analytic throughout $|t|<t^{(2)}$.
\end{proposition}
\begin{proof}
We verify here the avoidance conditions of Lemma~\ref{lem:two-PQ} at the
two arguments $t^2r(t)$ and $t$ needed in \eqref{eq:two-c-formula} and
\eqref{eq:two-g-exact}; we also verify that the product in the denominator
of \eqref{eq:two-c-formula} is nonzero.

For $0<|t|<t^{(2)}$, the arguments in $P_t(t^2r)$ and $Q_t(t^2r)$ are
$t^{2j+2}r=-t^{2j+1}v$. They cannot equal $r$, because $|t|<1$, and
equality with $s$ would give $t^{2j+1}v^2=1$, contrary to
Lemma~\ref{lem:two-v-bound}. The factors of $L_t$ and the additional
denominator $1-t^2w$ in $K_t$ are nonzero because
$|t^{2j+2}v|<1$ and $|t^{2j+3}v|<1$. In particular,
$P_t(t^2r)\neq0$.

For $P_t(t)$ and $Q_t(t)$ the arguments are $t^{2j+1}$. At $j=0$,
\[
 H_t(t)=1+b(t)+t^3=2(1-t)(1-t^2),
\]
which is nonzero in the disc. For $j\geq1$, putting $a=t^{(2)}$ gives
\[
 |H_t(t^{2j+1})-1|
 \leq(1+2a+2a^2+a^3)a^2+a^7<0.139.
\]
The remaining linear factors are nonzero as well. These estimates are
uniform on compact subsets and imply normal convergence of all products
and sums. They also show that the denominator in
\eqref{eq:two-c-formula} never vanishes, proving the claim.
\end{proof}

In the terminology of Flajolet and Sedgewick~\cite[Chapter~VI]{FS},
$\Delta$-regularity at a dominant point means analyticity in a disc
extending past that point after a narrow wedge with vertex there has been
removed.

\begin{proposition}\label{prop:two-Delta-regular}
The function $D^{(2)}$ is analytic in a neighbourhood of every point of
$|z|\leq z^{(2)}$ other than $z^{(2)}$, and in a dented neighbourhood of
$z^{(2)}$. Thus $z^{(2)}$ is its unique dominant singularity and $D^{(2)}$ is
$\Delta$-regular there.
\end{proposition}
\begin{proof}
The inverse $t(z)$ in \eqref{eq:t-of-z} has strictly positive
coefficients. Hence, for $|z|\leq z^{(2)}$,
\[
 |t(z)|\leq t(|z|)\leq t(z^{(2)})=t^{(2)}.
\]
On $|z|=z^{(2)}$, equality can hold only if all powers of $z$ have the same
argument, hence only at $z=z^{(2)}$. Proposition~\ref{prop:two-t-disc} gives
analytic continuation across every other point of the closed dominant
circle.

Near $t=t^{(2)}$, the two-variable expression $\widehat g^{(2)}(t,r)$ is
analytic, while $r(t)$ has only the square-root branch in
\eqref{eq:two-r-local}. Since $t(z)$ is analytic near $z^{(2)}$, this gives
continuation in an indented neighbourhood of $z^{(2)}$. A finite covering of
the remaining compact part of the circle gives the asserted
$\Delta$-domain.
\end{proof}

Define
\[
 \kappa^{(2)}=\frac{1-t^{(2)}}{2(t^{(2)})^2}
 \sqrt{-t^{(2)}\Delta'(t^{(2)})}.
\]
The coefficient of $\sqrt{1-z/z^{(2)}}$ in the local expansion of
$D^{(2)}$ is
\begin{equation}\label{eq:two-D1}
 D_1^{(2)}=z^{(2)}\,\partial_r\widehat g^{(2)}(t^{(2)},r^{(2)})
 \kappa^{(2)}
 \sqrt{\frac{1+t^{(2)}}{1-t^{(2)}}}
 =-0.3012581790492796850\ldots.
\end{equation}

\begin{theorem}\label{thm:two-asymptotic}
As $n\to\infty$,
\[
 d_n^{(2)}\sim C^{(2)}(3+2\sqrt2)^nn^{-3/2},
 \qquad
 C^{(2)}=-\frac{D_1^{(2)}}{2\sqrt\pi}
 =0.08498336328908425176\ldots.
\]
\end{theorem}
\begin{proof}
Combine Proposition~\ref{prop:two-Delta-regular}, the nonzero
coefficient above, and the square-root transfer theorem
\cite[Chapter~VI]{FS}.
\end{proof}

In particular,
\[
 \lim_{n\to\infty}(d_n^{(2)})^{1/n}
 =\lim_{n\to\infty}\frac{d_{n+1}^{(2)}}{d_n^{(2)}}
 =3+2\sqrt2.
\]

\subsection{Probabilistic consequences}\label{sec:two-probabilistic}
The sum decomposition leads to a different probabilistic
consequence here than in
the three-pattern class. Let $K_n^{(2)}$ be the number of
sum-indecomposable components of a uniformly random member of
$\Dtwo\cap \Sym_n$.

\begin{corollary}\label{cor:two-components}
Set
\[
 I^{(2)}(z)=1-\frac1{D^{(2)}(z)},\qquad
 D_0^{(2)}=D^{(2)}(\rho^{(2)}),\qquad
 \theta^{(2)}=I^{(2)}(\rho^{(2)})=1-\frac1{D_0^{(2)}}.
\]
Then, for every fixed $k\geq1$,
\begin{equation}\label{eq:two-component-limit}
 \mathbb P(K_n^{(2)}=k)
 \longrightarrow k(1-\theta^{(2)})^2(\theta^{(2)})^{k-1}.
\end{equation}
Equivalently, the limiting probability generating function is
\begin{equation}\label{eq:two-component-pgf}
 \frac{u(1-\theta^{(2)})^2}{(1-u\theta^{(2)})^2}.
\end{equation}
In particular,
\begin{align}
 \lim_{n\to\infty}\mathbb P(K_n^{(2)}=1)
 &=\frac1{(D_0^{(2)})^2},\label{eq:two-indecomp-limit}\\
 \lim_{n\to\infty}\mathbb E K_n^{(2)}
 &=\frac{1+\theta^{(2)}}{1-\theta^{(2)}}.
 \label{eq:two-component-mean}
\end{align}
Numerically,
\begin{align*}
 D_0^{(2)}&=1.38729662302128\ldots,&
 \theta^{(2)}&=0.279173621988513\ldots,\\
 \frac1{(D_0^{(2)})^2}&=0.519590667237160\ldots,&
 \frac{1+\theta^{(2)}}{1-\theta^{(2)}}&=1.77459324604257\ldots.
\end{align*}
The same limit $1/(D_0^{(2)})^2$ is the asymptotic proportion of
sum-indecomposable members of $\Dtwo$.
\end{corollary}
\begin{proof}
Both basis permutations of $\Dtwo$ are sum indecomposable, so $\Dtwo$ is
sum closed and unique sum decomposition gives
$D^{(2)}=1/(1-I^{(2)})$. Write $X=\sqrt{1-z/\rho^{(2)}}$. The local
expansion established above has the form
\[
 D^{(2)}(z)=D_0^{(2)}+D_1^{(2)}X+O(X^2),
\]
and hence
\begin{equation}\label{eq:two-I-local}
 I^{(2)}(z)=\theta^{(2)}
 +\frac{D_1^{(2)}}{(D_0^{(2)})^2}X+O(X^2).
\end{equation}
Marking components gives
\[
 D^{(2)}(z,u)=\frac1{1-uI^{(2)}(z)}.
\]
The coefficient of $X$ in this expression is
\[
 \frac{uD_1^{(2)}}{(D_0^{(2)})^2(1-u\theta^{(2)})^2}.
\]
At $u=1$ this reduces to $D_1^{(2)}$, since
$1-\theta^{(2)}=1/D_0^{(2)}$. The square-root transfer theorem, uniformly
for $u$ near $1$, therefore gives the limiting probability generating
function \eqref{eq:two-component-pgf}; coefficient extraction in $u$
gives \eqref{eq:two-component-limit}, and differentiation at $u=1$ gives
\eqref{eq:two-component-mean}. Finally, the coefficient of $X$ in
\eqref{eq:two-I-local} is $D_1^{(2)}/(D_0^{(2)})^2$, so
\[
 \frac{[z^n]I^{(2)}(z)}{d_n^{(2)}}
 \longrightarrow\frac1{(D_0^{(2)})^2}.
\]
\end{proof}

\subsection{Poles produced by dilation orbits and non-D-finiteness}\label{sec:two-poles}
The first sheet was enough for the coefficient asymptotic. To prove
non-D-finiteness we continue the product--sum expression across the
quadratic branch point.

Let $\mathcal S$ be the two-sheeted Riemann surface of the algebraic function $v(t)$ defined by
\begin{equation}\label{eq:two-v-surface}
 tv^2-b(t)v+t^2=0.
\end{equation}
On $\mathcal S$ put $r=-v/t$, so that the original branch
is the one with $v=t^2+O(t^3)$.

\begin{lemma}\label{lem:two-sheet-continuation}
The right-hand side of \eqref{eq:two-D-exact}, with $c$ given by
\eqref{eq:two-ctr}, defines a meromorphic function on $\mathcal S$. Near
$t=0$ on the original branch it agrees with the function defined by the
formal power series $D^{(2)}(t/(1+t)^2)$. In particular, this
function can be continued meromorphically along paths to either sheet,
apart from isolated poles.
\end{lemma}
\begin{proof}
Consider first a compact subset of $\mathcal S$. Its $t$-coordinates
satisfy $|t|\leq r_0<1$, and $v$ is bounded there. The arguments occurring
in $P_t(t),Q_t(t),P_t(t^2r)$, and $Q_t(t^2r)$ are respectively
\[
 t^{2i+1}\quad\hbox{and}\quad
 t^{2i+2}r=-t^{2i+1}v,
 \qquad i\geq0,
\]
and hence tend uniformly to zero. Uniformly in both variables,
\[
 \alpha_t(t^{2i}w)=1+O(r_0^{2i}),
 \qquad
 \beta_t(t^{2i}w)=O(r_0^{2i}).
\]
The tails of all four products and sums therefore converge normally.

For sufficiently large $i$, their denominators are uniformly close to
their nonzero values at $w=0$. Thus on any compact set only finitely many
initial factors can vanish. The union of all such denominator and
numerator-zero divisors is consequently locally finite. Factoring off
those finitely many initial terms shows that $P_t,Q_t,c$, and the final
expression are meromorphic across each of them. Theorem
\ref{thm:two-exact} identifies this expression with
$D^{(2)}(t/(1+t)^2)$ on a nonempty open subset of the original sheet. It
therefore provides a meromorphic continuation of that function from
$t=0$.

Finally, $\mathcal S$ is connected: its two sheets meet above the simple
zero $t^{(2)}$ of the discriminant. Removing a locally finite set from a
connected Riemann surface leaves continuation paths after arbitrarily
small detours; in particular, a loop in the $t$-plane around
$t^{(2)}$ lifts to a path that exchanges the two sheets, hence the continuation statement.
\end{proof}

For $N\geq1$, set
\[
 F_N(t)=H_t(t^{2N+1})=1+b(t)t^{2N}+t^{4N+3}.
\]
At a zero of $F_N$, the two values of $v$ are
\begin{equation}\label{eq:two-v-at-FN}
 v=-t^{2N+2},
 \qquad
 v=-t^{-(2N+1)}.
\end{equation}

\begin{proposition}\label{prop:two-orbit-pole}
Let $t_0$ be a simple zero of $F_N$ with $0<|t_0|<1$ and
$\Delta(t_0)\neq0$. On the sheet where $v=-t^{-(2N+1)}$, the
meromorphic continuation of
\[
 \widehat D^{(2)}(t)=D^{(2)}\left(\frac{t}{(1+t)^2}\right)
\]
has a pole of order three at $t=t_0$.
\end{proposition}
\begin{proof}
On this sheet, at $t=t_0$,
\[
 r=t^{-2N-2},
 \qquad
 s=t^{2N+1}.
\]
In $P_t(t^2r)$, the factor with index $N-1$ is evaluated at
$t^{2N}r=t^{-2}$. Here $\alpha_t$ has a simple zero, because
$1-t^2w=0$, while $\beta_t$ has a simple pole. As a function on the chosen
sheet, the vanishing factor is $1+t^{2N+1}v$; its zero is simple by the
hypothesis on $F_N$ and \eqref{eq:two-v-at-FN}. No other orbit point meets
a zero of $H_t$: the orbit points are $t^{2j-2N}$, while the roots are
$t^{-2N-2}$ and $t^{2N+1}$, and $|t|<1$ rules out equality of unequal
powers. The displayed factor is likewise the only zero contributed by
$L_t$.

All earlier factors are therefore finite and nonzero, and every later
term in $Q_t(t^2r)$ contains the vanishing factor. Thus $P_t(t^2r)$ has a
simple zero and $Q_t(t^2r)$ a simple pole. The quantity $\eta_t(r)$ is
regular, so \eqref{eq:two-ctr} shows that $c(t,r)$ has a pole of order
two.

In $P_t(t)$, the factor with index $N$ is evaluated at
$t^{2N+1}=s$, the uncancelled zero of $H_t$. Hence $P_t(t)$ has a
simple pole, and no other index meets either root. The sum $Q_t(t)$ has
at most a simple pole. Consequently $c(t,r)P_t(t)$ has a nonzero pole of
order three which cannot be cancelled by $Q_t(t)$. Equations
\eqref{eq:two-g-exact} and \eqref{eq:two-D-exact} prove the claim.
\end{proof}

\begin{proposition}\label{prop:two-infinite-poles}
For all sufficiently large $N$, the polynomial $F_N$ has a simple zero
$t_N$ with $|t_N|<1$. The zeros may be chosen distinct and tending to
$i$.
\end{proposition}
\begin{proof}
Put $t=i\exp(s/N)$. Uniformly for $s$ in compact sets,
\[
 F_N(ie^{s/N})\longrightarrow
 f_\epsilon(s)=1+\epsilon(3-3i)e^{2s}-ie^{4s},
 \qquad \epsilon=(-1)^N,
\]
along either parity subsequence. Indeed, $b(i)=3-3i$,
$i^{2N}=(-1)^N$, and $i^{4N+3}=-i$.

Writing $y=e^{2s}$, the limiting equation becomes
\[
 y^2+3\epsilon(1+i)y+i=0.
\]
It has the simple root
\[
 y_\epsilon=-\epsilon\frac{3-\sqrt7}{2}(1+i),
 \qquad
 |y_\epsilon|=\frac{3-\sqrt7}{\sqrt2}<1.
\]
Choose $s_\epsilon$ with $e^{2s_\epsilon}=y_\epsilon$; then
$\Re s_\epsilon<0$. Hurwitz's theorem, equivalently Rouch\'e's theorem on
a small circle, gives a simple zero
\[
 t_N=ie^{s_N/N},
 \qquad s_N=s_\epsilon+O(N^{-1}).
\]
Thus $|t_N|<1$ for large $N$ and $t_N\to i$. Passing, for instance, to
$N=2^k$ makes the zeros distinct. Since $\Delta(i)=7$, the two sheet
values in \eqref{eq:two-v-at-FN} are distinct at these zeros for all
sufficiently large $N$.
\end{proof}

\begin{corollary}\label{cor:two-orbit-pole-location}
For the zeros $t_N$ in Proposition~\ref{prop:two-infinite-poles}, put
$z_N=z(t_N)$. After passage to the distinct subsequence used there, the
$z_N$ are distinct poles of analytic continuations of $D^{(2)}$ and
\[
 z_N\longrightarrow z(i)=\frac12.
\]
More precisely, on either parity subsequence and with the choice of
$s_\epsilon$ in that proof,
\[
 z_N=\frac12-\frac{i s_\epsilon}{2N}+O(N^{-2}).
\]
\end{corollary}
\begin{proof}
Proposition~\ref{prop:two-orbit-pole} and Lemma
\ref{lem:two-sheet-continuation} make each point a pole of a continuation
starting from $t=0$ on the original sheet. The map
$z(t)=t/(1+t)^2$ identifies two values only when they are equal or
reciprocal; hence it is injective on $|t|<1$ and the chosen $z_N$ are
distinct. Since $z'(i)=-1/2$ and
$t_N=i(1+s_\epsilon/N+O(1/N^{2}))$, Taylor expansion gives both claims.
\end{proof}

\begin{theorem}\label{thm:two-non-D-finite}
The generating function $D^{(2)}$ is not D-finite, and $(d_n^{(2)})_{n\geq0}$
is not P-recursive.
\end{theorem}
\begin{proof}
If $D^{(2)}(z)$ were
D-finite, closure under rational substitution would make
$\widehat D^{(2)}(t)=D^{(2)}(t/(1+t)^2)$ D-finite~\cite{StanleyDF}.
Propositions~\ref{prop:two-orbit-pole} and
\ref{prop:two-infinite-poles} give infinitely many poles of analytic
continuations of $\widehat D^{(2)}$ in the finite $t$-plane,
and, as in the proof of Theorem~\ref{thm:three-non-D-finite}, this implies that $\widehat D^{(2)}$ is not D-finite.
\end{proof}

\end{document}